\documentclass[preprint,12pt]{elsarticle}

\usepackage{amssymb}

\usepackage{array}
\usepackage{multirow}

\usepackage{latexsym}
\usepackage{amsmath}
\usepackage{mathrsfs}
\usepackage[mathscr]{euscript}
\usepackage{graphicx}
\usepackage{url}
\usepackage{upgreek}
\usepackage{microtype}
\usepackage{stmaryrd}
\usepackage{mathtools}
\usepackage{xcolor}
\usepackage{caption}
\usepackage[all]{xy}
\usepackage[margin=1in]{geometry}
\usepackage{epsfig}
\usepackage{pgf,tikz}
\usepackage{multicol}

\DeclarePairedDelimiter\ceil{\lceil}{\rceil}
\DeclarePairedDelimiter\floor{\lfloor}{\rfloor}
\DeclareMathOperator{\ALPH}{Alph}

\usepackage{amsthm}

\usepackage{hyperref}
\usepackage{algpseudocode}
\usepackage{algorithm}

\newtheorem{theorem}{Theorem}[section]
\newtheorem{lemma}[theorem]{Lemma}
\newtheorem{proposition}[theorem]{Proposition}

\newtheorem{observation}[theorem]{Observation}
\newtheorem{definition}[theorem]{Definition}
\newtheorem{corollary}[theorem]{Corollary}
\newtheorem{example}[theorem]{Example}

\allowdisplaybreaks

\begin{document}

\begin{frontmatter}

\title{On words that nearly $\theta$-commute}

\author[lab1]{Anuran Maity}
 \ead{anuran.maity@gmail.com} 
 \author[lab2]{Kalpana Mahalingam}
 \ead{kmahalingam@iitm.ac.in}

\address[lab1]{Department of Mathematics,
	 Indian Institute of Technology Guwahati, Guwahati,
	   India}
\address[lab2]{Department of Mathematics, 
	 Indian Institute of Technology Madras, Chennai,
	   India}

\begin{abstract} 
The Hamming distance between two equal length words $\alpha, \beta$ is the number of positions where $\alpha$ and $\beta$ differ.
For  $x, y \in \Sigma^*$ and antimorphic involution $\theta$, $x$ $\theta$-commutes with $y$, if the Hamming distance between $xy$ and $\theta(y)x$ is zero. 
When the Hamming distance between $xy$ and $\theta(y)x$ attains its minimum non-zero value of one, then we say $x$ nearly $\theta$-commutes with $y$. This manuscript investigates properties of $x$ and $y$ such that the  Hamming distance between $xy$ and $\theta(y)x$ is one. We provide a complete characterization of such $x$ and $y$.
We introduce a binary relation $R_\theta$ on $\Sigma^*$, where $x R_\theta y$ holds if and only if $x$ nearly $\theta$-commutes with $y$.
Finally, for a given $y$, we collect all $x$ such that $x$ nearly $\theta$-commutes with $y$, and discuss various properties of this set.                 
\end{abstract}

\begin{keyword}
Hamming distance, $\theta$-commutativity.    
\end{keyword}
\end{frontmatter}

\section{Introduction}
For two equal length words $p$ and $q$, the Hamming distance between them, denoted as $H(p, q)$, measures the number of differing positions between $p$ and $q$ \cite{hamming1950error}.
The concept of Hamming distance has been widely used in various problems of combinatorics on words such as error detection (\cite{galil1990improved}), word transformations (\cite{anselmo2024computing, anselmo2020bad, beal2022checking,ilic2012index, wei2019characterization}), graph constructions (\cite{egecioglu2023fibonacci,KLAVZAR2012220}), etc. An interesting problem within the domain of word transformations is the determination of Hamming distances between a word and its conjugates.
A word $u$ is a \emph{conjugate} of a word $w$ if $w = xy$ and $u = yx$ for some words $x, y$. Shallit \cite{shallithamming} investigated all possible Hamming distances between a word and its conjugates and showed that the Hamming
distance between a word and its conjugate can never be one.
When the Hamming distance between $xy$ and $yx$ is zero, i.e.,  $H(xy, yx) = 0$, then the words $x$ and $y$ are said to \emph{commute} with each other. Generalizing this concept, Gabric \cite{gabric2022words} introduced the notion of \emph{almost commute} words. Two words $x$ and $y$ are said to \emph{almost commute} if the Hamming distance between $xy$ and $yx$ attains its minimum non-zero possible value of two.
From a theoretical standpoint, Gabric \cite{gabric2022words} characterized and enumerated all such pairs of words $x$ and $y$ that almost commute.

In DNA based computation, a word $y$ and its Watson-Crick complement $\theta_{WC}(y)$ encodes the same information and help to form intra/inter-molecular hybridizations where $\theta_{WC}$ is an antimorphic involution over $\{A, T, G, C\}^*$ with $\theta_{WC}(A)=T$ and $\theta_{WC}(C)=G$.
A function  $\theta:\Sigma^{*} \rightarrow \Sigma^{*}$ is said to be an \textit{antimorphic involution} if $\theta(uv)=\theta(v) \theta(u)$ for any $u, v \in \Sigma^*$, and $\theta^{2}$ is {the} identity mapping on $\Sigma^*$.
Inspired by these relations between $y$ and $\theta_{WC}(y)$, a lot of work has been done to extend the results of classical combinatorics on words in the antimorphic involution setting, for example, see \cite{CZ12},  \cite{kari2014pseudo},  \cite{kari2017disjunctivity}, \cite{kari2007involutively}, \cite{kari2008watson}, \cite{watson}, \cite{LKari2010}, \cite{kari2019state}, \cite{kari2009pseudoknot}, \cite{starosta2011theta}, etc.
Kari and Mahalingam \cite{watson} extended the idea of conjugates by defining the concept of \emph{$\theta$-conjugates} for an antimorphic involution $\theta$. 
A word $u$ is said to be a $\theta$-conjugate of a word $w$ if $w = xy$ and $u = \theta(y)x$ for some words $x, y$. 
Authors in \cite{watson}, \cite{mahalingam2025watson} and \cite{mahalingam2022counting} explored several properties of $w$ and $u$ such that $u$ is a $\theta$-conjugate of $w$. 
A word $x$ is said to \textit{$\theta$-commute} with a word $y$ if $xy = \theta(y)x$, i.e., $H(xy, \theta(y)x) = 0$. Authors in \cite{watson} characterized such words $x$ and $y$.
In this manuscript, we first observe that if $x$ does not $\theta$-commute with $y$, then the minimum non-zero Hamming distance between  $xy$ and $\theta(y)x$ is one. We then extend the works of Shallit \cite{shallithamming} and Gabric \cite{gabric2022words} to the setting of $\theta$-conjugates and study properties of words $x$ and $y$ that satisfy $H(xy, \theta(y)x) = 1$.
We first characterize such $x$ and $y$. Then, we introduce and study a binary relation $R_\theta$ on $\Sigma^*$, where $x R_\theta y$ holds if and only if $x$ nearly $\theta$-commutes with $y$. Finally, for a given $y$, we collect all $x$ such that $x$ nearly $\theta$-commutes with $y$, and discuss various properties of this set.
Our study is interesting from two perspectives. First, it is relevant from a theoretical standpoint.
Second, in the context of DNA-based computation, the strands $xy$ and $\theta_{WC}(y)x$ cannot be used simultaneously to avoid intermolecular hybridization (as explained in \cite{mahalingam2025watson}). Given that $H(xy, \theta_{WC}(y)x)=1$, i.e., $xy$ and $\theta_{WC}(y)x$ are almost same, determining whether either strand can be used for the computation requires a detailed understanding of the structural properties of these strands.

The manuscript is organized as follows:
Section \ref{Preliminary} presents some basic definitions and results that will be used throughout the text.
In Section \ref{maindefinition}, we define the notion of a word $x$ that nearly $\theta$-commutes with $y$ and discuss properties of such $x$ and $y$.
In Section \ref{relation}, we define and discuss a binary relation $R_\theta$ such that $xR_\theta y$ holds if $x$ nearly $\theta$-commutes with $y$. 
In Section \ref{Property}, we discuss various combinatorial properties of the set of all words $x$ that nearly $\theta$-commutes with a given $y$.
We end the manuscript with a few concluding remarks.

\section{Preliminaries}\label{Preliminary}

An \textit{alphabet} $\Sigma$ is a finite non-empty set of symbols. 
A \textit{word} over $\Sigma$ is a finite sequence of symbols from $\Sigma$. $\Sigma^{*}$ denotes the set of all words over $\Sigma$ including the \textit{empty word} $\lambda$ and $\Sigma^{+}=\Sigma^* \setminus \{\lambda\}$. 
The \textit{length of a word}  $w$ is denoted by $|w|$.
The \textit{reversal} of $w=a_{1}a_{2} \cdots a_{n}$ is defined to be the word $w^{R}=a_{n} \cdots a_{2} a_{1}$ where each $a_{i} \in \Sigma$. 
A word $w$ is said to be a \textit{palindrome} if $w=w^{R}$. 
A word $w\in \Sigma^+$ is called \textit{primitive} if it cannot be written as a power of another word; that is, $w=u^n$ implies $n=1$ and $w=u$.
 %
A word $u \in \Sigma^*$ is a \textit{factor} of $w$ if $w=xuy$ where $x, y \in \Sigma^*$. 
If $x= \lambda$ (resp. $y=\lambda$), then $u$ is a \textit{prefix} (resp. \textit{suffix}) of $w$. 
Also, if $x= \lambda$ and $u \neq w$ (resp. $y = \lambda$ and $u \neq w$), then $u$ is a \textit{proper prefix} (resp. \textit{proper suffix}) of $w$. The set of all prefixes (resp. suffixes and factors) of $w$ is denoted by $\text{Pref}(w)$ (resp. $\text{Suff}(w)$ and $\text{Fac}(w)$).
An integer $m \geq 1$ is a \textit{period} of a word $w = a_1 a_2 \cdots a_n$ where each $a_i \in \Sigma$, if $a_j = a_{j+m}$ for all $1 \leq j \leq n - m$.
For words $u$ and $w$, $|w|_u$ denotes the \textit{number of occurrences} of $u$ as a factor of $w$. We denote the set of all factors of $w$ of length $1$ by $\ALPH(w)$.
The \textit{Hamming distance}, denoted as $H(x, y)$, measures the number of differing positions between two words of equal length, $x$ and $y$. For instance, $H(\text{town}, \text{tree}) = 3$.
A word $u \in \Sigma^*$ is a \textit{conjugate} of $w \in \Sigma^*$ if
$w=\alpha \beta$ and $u=\beta \alpha$ for some $\alpha, \beta \in \Sigma^*$.
The set of all conjugates of $w$, denoted as $C(w)$, is the set $\{\beta \alpha: w=\alpha \beta \text{ where } \alpha, \beta \in \Sigma^*\}$.
For $x, y \in \Sigma^*$, if $x$ and $y$ \textit{commute}, then $xy=yx$, i.e., $H(xy, yx)=0$.
But interestingly, Shallit \cite{shallithamming} proved that the next
smallest value $H(xy, yx)$ can take on is $2$ and not $1$.
He showed the following.

\begin{theorem}\cite{shallithamming}\label{conalmc5++}
 For $x, y \in \Sigma^*$, $H(xy, yx) \neq 1$.
\end{theorem}

In a talk \cite{shallitsir}, Shallit introduced a similar concept of \emph{almost commute} words. Later, Gabric \cite{gabric2022words} defined it formally.
\begin{definition}\cite{gabric2022words}
   Two words $x$ and $y$ are said to \emph{almost commute} if $H(xy, yx) = 2$. 
\end{definition}

A function  $\theta$ is an \textit{antimorphism}  on $\Sigma^{*}$ if $\theta(uv)=\theta(v) \theta(u)$ holds for all $u, v \in \Sigma^*$, and is an \textit{involution} on $\Sigma$ if $\theta(\theta(a))=a$ holds for all $a \in \Sigma$. 
Throughout this text, we take $\theta$ as an antimorphic involution.

A word $w$ is called a \textit{$\theta$-palindrome} if $w = \theta(w)$. 
The set of all $\theta$-palindromes over $\Sigma^*$ is denoted by $\textup{P}_\theta$.
A word $u$ is a \textit{$\theta$-conjugate} of a word $w$ if
$w=xy$ and $u=\theta(y)x$ for some $x, y \in \Sigma^*$. The set of all $\theta$-conjugates of $w$ is denoted by $C_{\theta}(w)$.
For $x, y \in \Sigma^*$, $x$ is said to \textit{$\theta$-commute} with $y$ if $xy = \theta(y)x$, i.e., $H(xy, \theta(y)x) = 0$.

A grammar $G$ is said to be \textit{right-linear} (resp. \textit{left-linear}) if all production rules are of the form $A \rightarrow xB$ or $A \rightarrow x$ (resp. $A \rightarrow Bx$ or $A \rightarrow x$),  where $A, B$ are non-terminal symbols, and $x$ is a terminal symbol ($x$ can be empty). 
A \textit{regular grammar} is one that is either right-linear or left-linear. A Language $L$ is said to be \textit{regular} if there exists a regular grammar which generates $L$.

We also recall some basic results from the literature, which will be used multiple times in the rest of this manuscript.

\begin{lemma} \label{44lkc6+}\cite{Schutz62}
Let $u,v,w \in \Sigma^{+}$.
\begin{itemize}
     \item If $uv=vu$, then $u$ and $v$ are powers of a common primitive word. 
     \item If $uv=vw$, then for $k \geq 0$, $i\geq 1$, $x \in \Sigma^{+}$ and $y \in \Sigma^{*}$, $u=(xy)^i$, $v=(xy)^{k}x$, $w=(yx)^i$.
 \end{itemize}
\end{lemma}

\begin{proposition}\cite{watson}\label{ch4defctw32+2}
If  $xy = \theta(y)x$ for some $x, y \in \Sigma^+$, 
then $x = p(qp)^i$, $y = qp$ where $i \geq 0$, $p \in \textup{P}_\theta \cap \Sigma^+$ and $q \in \textup{P}_\theta \cap \Sigma^*$.
\end{proposition}

 The following result, known as Fine and Wilf theorem, illustrates a fundamental periodicity property of words. As usual, $gcd(n,m)$ denotes the greatest common divisor of $n$ and $m$.
\begin{theorem}\label{FineandWilf}\cite{Lothaire97}
    Let $u, v \in \Sigma^*$, $n=|u|$, $m=|v|$, and $d=gcd(n,m)$. If two powers $u^i$ and $v^j$ of $u$ and $v$ have a common prefix of length at least $n+m-d$, then $u$ and $v$ are powers of a common word. Moreover, the bound $n+m-d$ is optimal.
\end{theorem}

 For all other concepts, the reader is referred to  Lothaire \cite{Lothaire97}.

\section{Words that nearly $\theta$-commute}\label{maindefinition}
In this section we define the notion \textit{words that nearly $\theta$-commute.} It was shown in \cite{shallithamming} that the Hamming distance between $xy$ and $yx$ can never be 1 for all $x,y\in \Sigma^*$, but this is not the case for words $xy$ and $\theta(y)x$. The minimum non-zero Hamming distance between words $xy$ and $\theta(y)x$ is 1. We say $x$ nearly $\theta$-commutes with $y$ if $H(xy,\theta(y)x)=1$.  In this section, we characterize words $x$ and $y$ such that $x$ nearly $\theta$-commutes with $y$. 

We begin the section by illustrating with an example that the Hamming distance between $xy$ and $\theta(y)x$ for $x, y \in \Sigma^*$ can vary from $0$ to $|xy|$, i.e., $0 \leq H(xy, \theta(y)x) \leq |xy|$. 

\begin{example}\label{exm1}
Consider $\Sigma=\{a,b\}$ and $\theta$ be such that $\theta(a)=b$. 
\begin{enumerate}
     \item For $x=ba$ and $y= abba$, $H(xy, \theta(y)x)=H(baabba, baabba) =0< 6=|xy|$.
    \item For $x=ab$ and $y= b$, $H(xy,\theta(y)x)=H(abb, aab) = 1<3=|xy|$.
    \item For $x=ba$ and $y= ab$, $H(xy,\theta(y)x)=H(baab, abba)=4=|xy|$.
\end{enumerate}
\end{example}

From Example \ref{exm1}, it is clear that for $x, y \in \Sigma^*$, the next smallest possible Hamming distance between $xy$ and $\theta(y)x$ after $0$ is $1$.
If $ H(xy, \theta(y)x) = 1$, then the word $x$ \textit{``nearly" $\theta$-commutes} with $y$ in the sense that $xy$ and $\theta(y)x$ differ in exactly one position.
We now extend the idea of almost commute words and define the concept of words that nearly $\theta$-commute.

\begin{definition}\label{def1main}
For $x, y \in \Sigma^*$, $x$  \textit{nearly $\theta$-commutes} with $y$
if $ H(xy, \theta(y)x) = 1$.
\end{definition}

From Theorem \ref{conalmc5++}, we know that $H(xy, yx) \neq 1$ for any $x, y \in \Sigma^*$. 
This implies that if $y \in \textup{P}_\theta$, then $H(xy, \theta(y)x) \neq 1$. 
Thus, we have the following observation from Theorem \ref{conalmc5++} and Definition \ref{def1main}.
\begin{observation}\label{obseryPtheta}
  If $x$ nearly $\theta$-commutes with $y$, then $y \notin \textup{P}_\theta$. In other words,  if $y \in \textup{P}_\theta$, then there does not exist any $x \in \Sigma^*$ such that $x$ nearly $\theta$-commutes with $y$.
\end{observation}

 We now show that, by inspecting the maximal length of a shared prefix of $xy$ and $\theta(y)x$, in some cases,  it is possible to establish that $x$ nearly $\theta$-commutes with $y$, without looking into the structures of $x$ and $y$.

\begin{proposition}\label{c6fwko+}
Let $x, y \in \Sigma^+$ such that $|x|\leq |y|$. If $xy=uav$ and $\theta(y)x = ubq$ where $ |u|=|y|+\floor{\frac{|x|}{2}}$ and $a, b \in \Sigma$ with $a\neq b$, $b\neq \theta(b)$ then $H(xy,\theta(y)x) = 1$.
\end{proposition}
\begin{proof}
Let $xy=uav$ and $\theta(y)x = ubq$ with $a\neq b$, $b\neq \theta(b)$ and $a,b\in \Sigma$. Thus,
$H(xy,\theta(y)x) \ge 1$.
We need to show that $x$ nearly $\theta$-commutes with $y$ i.e.,
$$H(xy,\theta(y)x) =H(uav, ubq) = 1, \text{ or equivalently, } H(v,q)=0.$$
Since, $\theta(y)x = ubq$ and  $|u|=|y|+\floor{\frac{|x|}{2}}$, for $x=x_1x_2$, $x_1, x_2\in \Sigma^*$ and $|x_1| = \floor{\frac{|x|}{2}}$ we have $u= \theta(y)x_1$ and $x_2 = bq$ which implies that 
\begin{equation}\label{primary_equ}
    xy = uav = \theta(y)x_1av.
\end{equation}
As $|x|\le |y|$, for  $y=y_1 y_2$ with $y_1, y_2 \in \Sigma^*$ and $|y_2|=|x|$ we have from Equation (\ref{primary_equ}),
$\theta(y_2) = x$, $y_1 =\theta(y_1)$ and $y_2 =\theta(x) = x_1av$. 
Then as $x_2=bq$, we have
\begin{equation}\label{secondary_equ}
    \theta(x) =\theta(x_2)\theta(x_1) = \theta(q)\theta(b)\theta(x_1) = x_1av. 
\end{equation}
From Equation (\ref{secondary_equ}), we now have two cases (see Figure \ref{Fig_rem_y>x}):
\begin{enumerate}
    \item If $|x|$ is even, then $|x_1| = |x_2|= |av|=|bq|$ and $x_1 = \theta(q)\theta(b) = \theta(v)\theta(a)$ which implies that $a =b$, which is a contradiction to our assumption that $a\neq b$.
    \item If $|x|$ is odd, then $|x_1| =|v| =|q|$ and $x_1 = \theta(q) = \theta(v)$ and $a =\theta(b)$ which implies that 
    $$H(v,q) = 0, \text{ i.e., } H(xy,\theta(y)x) = 1.$$
\end{enumerate}

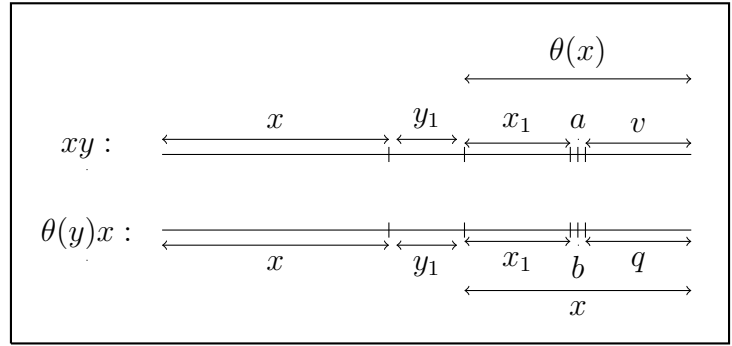
\begin{figure}[h]
    \centering
        \begin{tikzpicture}
        
        \draw[thick,-] (-2,-2.5) -- (7.5,-2.5) -- (7.5,2) -- (-2,2) -- (-2,-2.5);
        
        \draw[] (-1,-0.2) -- (-1,-0.2) node[midway, above] {$xy : $};
        \draw[] (-1,-1.4) -- (-1,-1.4) node[midway, above] {$\theta(y)x : $};
        
        \draw[-] (0,0) -- (7,0); 
        \draw[-] (3,0.1) -- (3,-0.1);
        \draw[<->] (0,0.2) -- (3,0.2) node[midway, above] {$x$};
        \draw[<->] (3.1,0.2) -- (3.9,0.2) node[midway, above] {$y_1$};
        \draw[-] (4,0.1) -- (4,-0.1);
        \draw[<->] (4,1) -- (7,1) node[midway, above] {$\theta(x)$};

        \draw[<->] (4,0.15) -- (5.4,0.15) node[midway, above] {$x_1$};
        \draw[] (5.5,0.2) -- (5.5,0.2) node[midway, above] {$a$};
        \draw[<->] (5.6,0.15) -- (7,0.15) node[midway, above] {$v$};
        \draw[-] (5.4,0.1) -- (5.4,-0.1);
        \draw[-] (5.5,0.1) -- (5.5,-0.1);
        \draw[-] (5.6,0.1) -- (5.6,-0.1);

        \draw[-] (0,-1) -- (7,-1);
        \draw[-] (3,-0.9) -- (3,-1.1);
        \draw[<->] (0,-1.2) -- (3,-1.2) node[midway, below] {$x$};
        \draw[<->] (3.1,-1.2) -- (3.9,-1.2) node[midway, below] {$y_1$};
        \draw[-] (4,-0.9) -- (4,-1.1);
        \draw[<->] (4,-1.8) -- (7,-1.8) node[midway, below] {$x$};

        \draw[-] (5.4,-0.9) -- (5.4,-1.1);
        \draw[-] (5.5,-0.9) -- (5.5,-1.1);
        \draw[-] (5.6,-0.9) -- (5.6,-1.1);
        \draw[<->] (4,-1.15) -- (5.4,-1.15) node[midway, below] {$x_1$};
        \draw[] (5.5,-1.2) -- (5.5,-1.2) node[midway, below] {$b$};
        \draw[<->] (5.6,-1.15) -- (7,-1.15) node[midway, below] {$q$};

        \end{tikzpicture}
        \caption{Illustration of the proof of Proposition \ref{c6fwko+}.}
        \label{Fig_rem_y>x}
\end{figure}

\end{proof}

\begin{proposition}\label{fwxgyc6+}
Let $x, y \in \Sigma^+$ such that $|x| > |y|$ and $|x| = k|y|+r$ where $0 \leq r < |y|$ and $k\ge 1$. 
\begin{itemize}
    \item If $r \neq 0$ and $xy$, $\theta(y)x$ agree on a prefix of length $|xy| - \ceil{\frac{r}{2}}$ but disagree at position $|xy| - \ceil{\frac{r}{2}} +1$, then $H(xy, \theta(y)x)=1$. 
    
    \item If $r=0$ and $xy$, $\theta(y)x$ agree on a prefix of length $|x| + \floor{\frac{|y|}{2}} $ but disagree at position $|x| + \floor{\frac{|y|}{2}}+1$, then $H(xy, \theta(y)x)= 1$. 
\end{itemize}
\end{proposition}

\begin{proof}
    Since $|x|>|y|$ and $|x| =k|y|+r$ with $k\ge 1$, let $x =y^{(1)} y^{(2)} \cdots y^{(k)} x' $ where $|y^{(i)}| =|y|$ for each $1 \leq i \leq k$, and $|x'|=r$. Then, $xy= y^{(1)} y^{(2)} \cdots y^{(k)} x' y   \text{ and } \theta(y)x = \theta(y)  y^{(1)} y^{(2)} \cdots y^{(k)} x'$ (see Figure \ref{Fig_rem_1_y>x}).

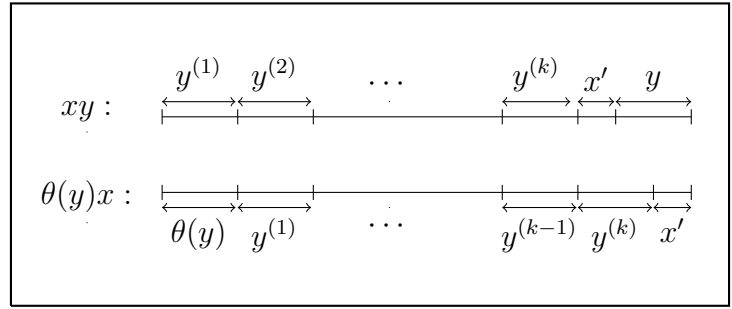
\begin{figure}[h]
    \centering
        \begin{tikzpicture}
        
        \draw[thick,-] (-2,-2.5) -- (7.5,-2.5) -- (7.5,1.5) -- (-2,1.5) -- (-2,-2.5);

        \draw[] (-1,-0.2) -- (-1,-0.2) node[midway, above] {$xy : $};
        \draw[] (-1,-1.4) -- (-1,-1.4) node[midway, above] {$\theta(y)x : $};

        \draw[-] (0,0) -- (7,0); 
        \draw[-] (0,0.1) -- (0,-0.1);
        \draw[-] (1,0.1) -- (1,-0.1);
        \draw[-] (2,0.1) -- (2,-0.1);
        \draw[-] (4.5,0.1) -- (4.5,-0.1);
        \draw[-] (5.5,0.1) -- (5.5,-0.1);
        \draw[-] (6,0.1) -- (6,-0.1);
        \draw[-] (7,0.1) -- (7,-0.1);
        \draw[<->] (0,0.2) -- (0.97,0.2) node[midway, above] {$y^{(1)}$};
        \draw[<->] (1,0.2) -- (1.97,0.2) node[midway, above] {$y^{(2)}$};
        \draw[] (3,0.2) -- (3,0.2) node[midway, above] {$\cdots$};
        \draw[<->] (4.5,0.2) -- (5.4,0.2) node[midway, above] {$y^{(k)}$};
        \draw[<->] (5.5,0.2) -- (5.98,0.2) node[midway, above] {$x'$};
        \draw[<->] (6,0.2) -- (7,0.2) node[midway, above] {$y$};

        \draw[-] (0,-1) -- (7,-1); 
        \draw[-] (0,-0.9) -- (0,-1.1);
        \draw[-] (1,-0.9) -- (1,-1.1);
        \draw[-] (2,-0.9) -- (2,-1.1);
        \draw[-] (4.5,-0.9) -- (4.5,-1.1);
        \draw[-] (5.5,-0.9) -- (5.5,-1.1);
        \draw[-] (6.5,-0.9) -- (6.5,-1.1);
        \draw[-] (7,-0.9) -- (7,-1.1);
        \draw[<->] (0,-1.2) -- (0.97,-1.2) node[midway, below] {$\theta(y)$};
        \draw[<->] (1,-1.2) -- (1.97,-1.2) node[midway, below] {$y^{(1)}$};
        \draw[] (3,-1.2) -- (3,-1.2) node[midway, below] {$\cdots$};
        \draw[<->] (4.5,-1.2) -- (5.48,-1.2) node[midway, below] {$y^{(k-1)}$};
        \draw[<->] (5.5,-1.2) -- (6.48,-1.2) node[midway, below] {$y^{(k)}$};
        \draw[<->] (6.5,-1.2) -- (7,-1.2) node[midway, below] {$x'$};
        \end{tikzpicture}
        \caption{Illustration of the proof of Proposition \ref{fwxgyc6+}.}
        \label{Fig_rem_1_y>x}
\end{figure}

We now have the following cases.

\begin{itemize}
    \item  {\underline{Case 1~:~}} $r \neq 0$. 
    
    Since $xy$ and $\theta(y)x$ agree on a prefix of length $|xy| - \ceil{\frac{r}{2}}$, we have
    $xy= (\theta(y))^k x' y$ and $\theta(y)x =  (\theta(y))^k \theta(y) x'$.
    Consider $y=y_1 y_2$ for some $y_1, y_2 \in \Sigma^*$ where $|y_2|=|x'|$.
    Then,  $xy= (\theta(y))^k x' y_1 y_2$ and $\theta(y)x =  (\theta(y))^k \theta(y_2) \theta(y_1) x'$. 
    Since $xy$ and $\theta(y)x$ agree on a prefix of length $|xy| - \ceil{\frac{r}{2}}$, we have 
    $x'=\theta(y_2)$ and $y_1 = \theta(y_1)$.
    Now, we have the following cases based on $r$:
     \begin{enumerate}
         \item If $r$ is even, then $|x'|=r \geq 2 $. Consider $x'= x_1 a b x_2$ for some $a, b \in \Sigma$, $x_1, x_2 \in \Sigma^*$ where $|x_1|=|x_2|$. Then 
          $xy= (\theta(y))^k x' y_1 \theta(x') = (\theta(y))^k x' y_1 \theta(x_2) \theta(b) \theta(a) \theta(x_1)$ and $\theta(y)x =  (\theta(y))^k x' y_1 x' = (\theta(y))^k x' y_1 x_1 ab x_2$.
          Now, $xy$ and $\theta(y)x$ disagree at position  $|xy| - {\frac{r}{2}} +1$ implies $\theta(a) \neq b$, i.e., $a \neq \theta(b)$. But $xy$ and $\theta(y)x$ agree on a prefix of length $|xy| - {\frac{r}{2}}$  implies $\theta(b)=a$, which is a contradiction as $a \neq \theta(b)$.

         \item  If $r$ is odd, then $x'= x_1 a x_2$ for some $a \in \Sigma$, $x_1, x_2 \in \Sigma^*$ where $|x_1|=|x_2|$. Then 
          $xy=  (\theta(y))^k x' y_1 \theta(x_2) \theta(a) \theta(x_1)$ and $\theta(y)x = (\theta(y))^k x' y_1 x_1 a x_2$. Since  $xy$, $\theta(y)x$ agree on a prefix of length $|xy| - \ceil{\frac{r}{2}}$ but disagree at position $|xy| - \ceil{\frac{r}{2}} +1$, we have $\theta(x_2)=x_1$ and $ \theta(a) \neq a$. This implies $\theta(x_1)=x_2$. Thus,  $H(xy, \theta(y)x)  = H(\theta(a), a)=1$.

     \end{enumerate}

    \item   {\underline{Case 2~:~}} $r = 0$. 

    Then $x'=\lambda$. Since $xy$, $\theta(y)x$ agree on a prefix of length $|x| + \floor{\frac{|y|}{2}}$,  $\theta(y)=y^{(i)}$ for all $1 \leq i \leq k$. Then, $xy= (\theta(y))^k y$ and $\theta(y)x =  (\theta(y))^k \theta(y)$. 
    Since $xy$, $\theta(y)x$ agree on a prefix of length $|x| + \floor{\frac{|y|}{2}}$ but disagree at position $|x| + \floor{\frac{|y|}{2}} +1$, similar to Case 1 one can show that $H(xy, \theta(y)x) = H(y, \theta(y))=1$.

\end{itemize}
\end{proof}

To characterize words $x$ and $y$ such that $x$ nearly $\theta$-commutes with $y$, we need the following result.
\begin{lemma}\label{lem1c619++}
For $w \in \Sigma^+$, $H(w, \theta(w))=1$ if and only if $w=\alpha a \theta(\alpha)$ for some $ \alpha \in \Sigma^*$ and $a \in \Sigma$ with  $\theta(a)\neq a$.
\end{lemma}
\begin{proof}
Let $w=a_1 a_2 \cdots a_n$ where $n \geq 1$ and each $a_i \in \Sigma$. We first show that if $H(w, \theta(w))=1$, then $n$ is odd. 
Assume by contradiction that $|w|$ is even. Then, for some $x, y \in \Sigma^*$ with $|x|=|y|$, $w=xy$ and $\theta(w)=\theta(y)\theta(x)$. 
\begin{itemize}
    \item If $x \neq \theta(y)$,  then $y \neq \theta(x)$ and  $H(w, \theta(w))\geq 2$, which is a contradiction. 
    \item If $x=\theta(y)$, then $y=\theta(x)$ and $H(w, \theta(w))=0$, which is a contradiction.
\end{itemize}
 Thus, $n$ is odd.
Using induction on $n$, we now prove that if $H(w, \theta(w))=1$, then $w=\alpha a \theta(\alpha)$ for some $ \alpha \in \Sigma^*$ and $a \in \Sigma$ with  $\theta(a)\neq a$. 
For $n=1$, $w=a_1$. Then, $H(w, \theta(w))=H(a_1,\theta(a_1)) = 1$ implies $a_1 \neq \theta(a_1)$. Thus, $w=a_1$ with $a_1 \neq \theta(a_1)$. Assume the statement be true for all words $w$ of odd length less than $n$. 
We now prove the result for $w$ of length $n \geq 3$.
 Let $u= a_2 \cdots a_{n-1}$. Then, $|u|$ is odd. Now, $H(w, \theta(w))=1$ implies $H(a_1 u a_n, \theta(a_n) \theta(u) \theta(a_1))=1$. 
 If $a_1 \neq \theta(a_n)$, then $a_n \neq \theta(a_1)$ which implies $H(w, \theta(w)) \geq 2$,  a contradiction. Thus, $a_1 = \theta(a_n)$. Then, 
\begin{align*}
    H(a_1 u a_n, \theta(a_n) \theta(u) \theta(a_1))=1 &\implies H(a_1 u \theta(a_1), a_1 \theta(u) \theta(a_1))=1\\
    &\implies H(u, \theta(u))=1.
\end{align*}
Then, by induction $u=\beta b \theta(\beta)$ for some $ \beta \in \Sigma^*$ and $b \in \Sigma$ with  $\theta(b)\neq b$. Therefore, $w= a_1 \beta b \theta(\beta) \theta(a_1) $, i.e., $w= a_1 \beta b \theta(a_1 \beta) $ where $\theta(b)\neq b$.\\
The converse is straightforward. 
\end{proof}

 In the following we give necessary and sufficient conditions on words $x$ and $y$ such that $x$ nearly $\theta$-commutes with $y$. 
\begin{theorem}\label{mainstrucrure}
For $x \in \Sigma^*, y \in \Sigma^+$, $x$ nearly $\theta$-commutes with $y$ if and only if
one of the following holds:
\begin{enumerate}
    \item $x=\theta(y)^i z y^j $ and $y=tz$  
    \item $x=\theta(y)^i \theta(z)$ and $ y= tz$ 
    \item $x=\theta(y)^i t y^j $ and $y=zt$
\end{enumerate}
where $i, j \geq 0$, $z \in \Sigma^+$, $H(z, \theta(z))=1$, $t \in \Sigma^* \cap \textup{P}_\theta$.
\end{theorem}

\begin{proof}
For $x \in \Sigma^*$ and $ y \in \Sigma^+$, let $x$ nearly $\theta$-commutes with $y$. Then, for $w=xy$ and  $u=\theta(y)x$, we have $H(w, u)=1$. Comparing $|x|$ and $|y|$, we have the following cases:
\begin{itemize}
\item \textbf{Case 1:}  $|y|\geq|x|$ : 

For $y_1, y_2 \in \Sigma^*$, let $y=y_1 y_2$ where $|y_2|=|x|$ (see Figure \ref{Fig_y>x}). Then,
        $ w = x y_1 y_2 \text{ and } u =\theta(y_2) \theta(y_1) x.$
        We now have the following cases:
        \begin{itemize}
            \item If $x \neq \theta(y_2)$, then as  $H(w, u)=1$, we have $y_1=\theta(y_1)$ and $y_2 =x$. This implies $x \neq \theta(x)$ and $H(w, u) = H(xy_1x, \theta(x)\theta(y_1)x) =H(x, \theta(x))=1$. Thus, $y=y_1y_2 =y_1x$ where $y_1 \in \textup{P}_\theta$ and $H(x,\theta(x)) =1$. 
        
            \item If $ y_1 \neq \theta(y_1)$, then by a similar argument we can show that $y=y_1x$ where $H(y_1,\theta(y_1)) =1$ and $x \in \textup{P}_\theta$. 
            \item The case when $ y_2 \neq x$ also concludes that $y=y_1 \theta(x)$ where $y_1\in \textup{P}_\theta$ and $H(x,\theta(x)) =1$.
              
        \end{itemize}

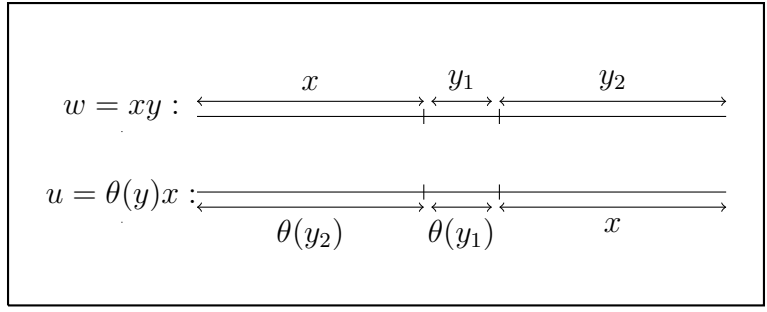
\begin{figure}[h]
    \centering
        \begin{tikzpicture}

        \draw[thick,-] (-2.5,-2.5) -- (7.5,-2.5) -- (7.5,1.5) -- (-2.5,1.5) -- (-2.5,-2.5);

        \draw[] (-1,-0.2) -- (-1,-0.2) node[midway, above] {$w=xy : $};
        \draw[] (-1,-1.4) -- (-1,-1.4) node[midway, above] {$u=\theta(y)x : $};
        
        \draw[-] (0,0) -- (7,0); 
        \draw[-] (3,0.1) -- (3,-0.1);
        \draw[<->] (0,0.2) -- (3,0.2) node[midway, above] {$x$};
        \draw[<->] (3.1,0.2) -- (3.9,0.2) node[midway, above] {$y_1$};
        \draw[-] (4,0.1) -- (4,-0.1);
        \draw[<->] (4,0.2) -- (7,0.2) node[midway, above] {$y_2$};

        \draw[-] (0,-1) -- (7,-1);
        \draw[-] (3,-0.9) -- (3,-1.1);
        \draw[<->] (0,-1.2) -- (3,-1.2) node[midway, below] {$\theta(y_2)$};
        \draw[<->] (3.1,-1.2) -- (3.9,-1.2) node[midway, below] {$\theta(y_1)$};
        \draw[-] (4,-0.9) -- (4,-1.1);
        \draw[<->] (4,-1.2) -- (7,-1.2) node[midway, below] {$x$};

        \end{tikzpicture}
        \caption{Illustration of the case $|y|\geq |x|$.}
        \label{Fig_y>x}
\end{figure}

\item \textbf{Case 2:} $|y|<|x|$ : 

For $x_1, x_2 \in \Sigma^+$, let $x=x_1 x_2$ where $|x_1|=|y|$. Then,
             $ w = x_1 x_2 y \text{ and } u =\theta(y) x_1 x_2$.
            We now have the following cases:
            \begin{itemize}
                \item \textbf{Case 2.1:} $x_1 \neq \theta(y)$ (see Figure \ref{Fig_y<x_1}): 
                
                Since  $H(w, u)=1$, we have $x_2 y = x_1 x_2$. Then by Lemma \ref{44lkc6+}, $x_1 = (pq)^i$, $x_2 = (pq)^kp$ and $y=(qp)^i$ for some $k \geq 0$, $i \geq 1$, $p \in \Sigma^+$ and $q \in \Sigma^*$. Thus,
                $w= (pq)^i (pq)^kp (qp)^i = (pq)^i (pq)^{k+i}p$ and $u= (\theta(p) \theta(q))^i (pq)^i (pq)^kp = (\theta(p) \theta(q))^i (pq)^{k+i}p$. 
                
                Now,  $x_1 \neq \theta(y)$  implies $(pq)^i \neq (\theta(p) \theta(q))^i$, i.e., $pq \neq \theta(p) \theta(q)$.
                Then as $H(w, u)=1$, we have $i=1$. Also, $H(w, u)=1$ implies either $\theta(p)\neq p$ and $\theta(q)=q$, or $\theta(p)=p$ and $\theta(q)\neq q$.
                
                If $\theta(p)\neq p$ and $\theta(q)=q$, then as $H(w, u)=1$, we have $H(p, \theta(p))=1$. 
                Thus, $x = (pq)^{k+1}p$ and $y=qp$ where $H(p, \theta(p))=1$ and $q \in \textup{P}_\theta$.
                
               
                If $\theta(q)\neq q$ and $\theta(p)=p$, then as $H(w, u)=1$, we have $H(q, \theta(q))=1$.
                Thus, $x = (pq)^{k+1}p$ and $y=qp$ where $H(q, \theta(q))=1$ and $p \in \textup{P}_\theta$.

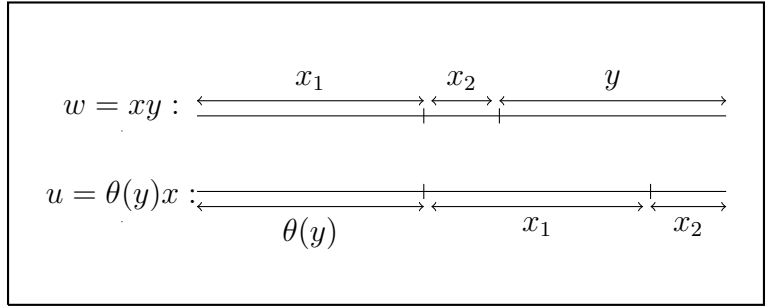
\begin{figure}[h]
    \centering
        \begin{tikzpicture}

         \draw[thick,-] (-2.5,-2.5) -- (7.5,-2.5) -- (7.5,1.5) -- (-2.5,1.5) -- (-2.5,-2.5);
         
        \draw[] (-1,-0.2) -- (-1,-0.2) node[midway, above] {$w=xy : $};
        \draw[] (-1,-1.4) -- (-1,-1.4) node[midway, above] {$u=\theta(y)x : $};

        \draw[-] (0,0) -- (7,0);   
        \draw[-] (3,0.1) -- (3,-0.1);
        \draw[<->] (0,0.2) -- (3,0.2) node[midway, above] {$x_1$};
        \draw[<->] (3.1,0.2) -- (3.9,0.2) node[midway, above] {$x_2$};
        \draw[-] (4,0.1) -- (4,-0.1);
        \draw[<->] (4,0.2) -- (7,0.2) node[midway, above] {$y$};

        \draw[-] (0,-1) -- (7,-1);
        \draw[-] (3,-0.9) -- (3,-1.1);
        \draw[<->] (0,-1.2) -- (3,-1.2) node[midway, below] {$\theta(y)$};
        \draw[<->] (3.1,-1.2) -- (5.9,-1.2) node[midway, below] {$x_1$};
        \draw[-] (6,-0.9) -- (6,-1.1);
        \draw[<->] (6,-1.2) -- (7,-1.2) node[midway, below] {$x_2$};

        \end{tikzpicture}
        \caption{Illustration of the proof of Theorem \ref{mainstrucrure} in the Case 2.1.}
        \label{Fig_y<x_1}
\end{figure}

 \item \textbf{Case 2.2:} $x_2y \neq x_1 x_2$ (see Figure \ref{Fig_y<x_2}):
             
             Since  $H(w, u)=1$, we have $x_1 = \theta(y)$ and 
              $$  w = \theta(y)  x_2 y, ~~ u =\theta(y) \theta(y) x_2,~~~H(x_2y,\theta(y)x_2)=1.$$
            We now have the following cases:
            
            If $|y|\geq |x_2|$, then similar to Case 1,  we have the following forms of $x$ and $y$.
            \begin{enumerate}
              \item 
              $x =\theta(y)  x_2$ and $y= y' x_2  \text{ where }  \theta(y') = y', H(x_2, \theta(x_2))=1$.
              \item
              $x= \theta(y) x_2$ and $y=y' x_2  \text{ where } \theta(x_2)=x_2$,  and $H(y', \theta(y'))=1$.
              \item
              $x= \theta(y) x_2$ and $y=y' \theta(x_2)  \text{ where } \theta(y')=y'$,  and $H(x_2, \theta(x_2))=1$.
            \end{enumerate}
            If $|y|<|x_2|$, then similar to Case 2.1 and initial part of Case 2.2, we have the following forms of $w$.
            \begin{enumerate}
                 \item 
               $x = \theta(p_1)q_1  (p_1q_1)^{m+1}p_1 $ and $y = q_1p_1
               \text{ where } H(p_1, \theta(p_1))=1, \theta(q_1)=q_1, m \geq 0$.
                \item 
              $x =p_1\theta(q_1) (p_1q_1)^{m+1}p_1$ and $y=q_1p_1$
               where $H(q_1, \theta(q_1))=1$, $\theta(p_1)=p_1$, $m \geq 0$. 
                \item 
               $w = \theta(y) \theta(y)  x_2' y, u =  \theta(y) \theta(y) \theta(y)  x_2'
              \text{ where } x_2 = \theta(y) x_2'$ for some $x_2' \in \Sigma^+$.
            \end{enumerate}

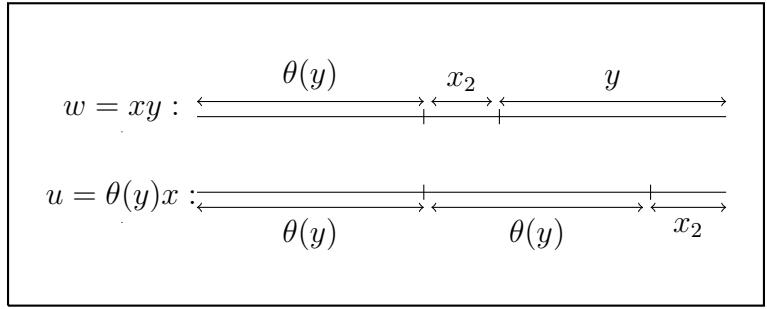
\begin{figure}[h]
    \centering
        \begin{tikzpicture}
         \draw[thick,-] (-2.5,-2.5) -- (7.5,-2.5) -- (7.5,1.5) -- (-2.5,1.5) -- (-2.5,-2.5);
        
        \draw[] (-1,-0.2) -- (-1,-0.2) node[midway, above] {$w=xy : $};
        \draw[] (-1,-1.4) -- (-1,-1.4) node[midway, above] {$u=\theta(y)x : $};

        \draw[-] (0,0) -- (7,0);   
        \draw[-] (3,0.1) -- (3,-0.1);
        \draw[<->] (0,0.2) -- (3,0.2) node[midway, above] {$\theta(y)$};
        \draw[<->] (3.1,0.2) -- (3.9,0.2) node[midway, above] {$x_2$};
        \draw[-] (4,0.1) -- (4,-0.1);
        \draw[<->] (4,0.2) -- (7,0.2) node[midway, above] {$y$};

        \draw[-] (0,-1) -- (7,-1);
        \draw[-] (3,-0.9) -- (3,-1.1);
        \draw[<->] (0,-1.2) -- (3,-1.2) node[midway, below] {$\theta(y)$};
        \draw[<->] (3.1,-1.2) -- (5.9,-1.2) node[midway, below] {$\theta(y)$};
        \draw[-] (6,-0.9) -- (6,-1.1);
        \draw[<->] (6,-1.2) -- (7,-1.2) node[midway, below] {$x_2$};

        \end{tikzpicture}
        \caption{Illustration of the proof of Theorem \ref{mainstrucrure} in the Case 2.2.}
        \label{Fig_y<x_2}
\end{figure}

            If $w = \theta(y) \theta(y)  x_2' y$ and $  u =  \theta(y) \theta(y) \theta(y)  x_2'$, then let us continue the process depending on whether $|y|\geq |x_2'|$ or $|y|<|x_2'|$. Since $|w|$ is finite, after some finite number of steps, $w$ and $u$ will be in the forms $ (\theta(y))^{m_1}  g y$ and $  (\theta(y))^{m_1} \theta(y) g$, respectively, where $ g \in \Sigma^*, m_1 \geq 1, |y|\geq |g|$.
          Then, using Case $1$, it can be shown that 
          $x=(\theta(z)t)^{m_1} z$ and $y=tz$ or $x=(t\theta(z))^{m_1} t$ and $y=zt$ or  $x=(zt)^{m_1} z$ and $y=t\theta(z)$  where $H(z, \theta(z))=1$, $t \in \textup{P}_\theta$.
            \end{itemize}
\end{itemize}

Now, collecting all the results obtained in Case 1 and Case 2, with all their  subcases, we have that if $y$ is in form $tz$, then $x$ is either in the form $z$ or $(zt)^{l}z$ or $(\theta(z)t)^{l}z$ or $(\theta(z)t)^{l} (zt)^{l_1}z$ or $\theta(z)$ or $(\theta(z)t)^{l}\theta(z)$
where $l, l_1 \geq 1$. Thus, if $y$ is in form $tz$, then $x$ is in form either 
$(\theta(z)t)^{r_1} (zt)^{r_2} z$ or $(\theta(z)t)^{r_1} \theta(z)$, i.e., 
either $\theta(y)^{r_1} z y^{r_2} $ or $\theta(y)^{r_1} \theta(z)$ 
where $r_1, r_2 \geq 0$ (proving cases 1 and 2 of the statement).
Also, if $y$ is in form $zt$, then $x$ is in the form $t$ or $(tz)^{l}t$ or $(t\theta(z))^{l}t$ or $(t \theta(z))^{l} (tz)^{l_1}t$ where $l, l_1 \geq 1$. 
Thus, if  $y$ is in form $zt$, then $x$ is in the form 
$(t \theta(z))^{l'} (tz)^{l_1'}t$, i.e., $\theta(y)^{l'} t y^{l_1'} $ 
where $l', l_1' \geq 0$ (proving case 3 of the statement). \\

Conversely, let $x=\theta(y)^i z y^j $ and $y=tz$  or $x=\theta(y)^i \theta(z)$ and $ y= tz$ or $x=\theta(y)^i t y^j $ and $y=zt$ where $i, j \geq 0$, $H(z, \theta(z))=1$, $t \in \textup{P}_\theta$. Then 
$x=(\theta(z)t)^i (zt)^j z$ and $y=tz$ or $x=(\theta(z)t)^i \theta(z)$ and $ y= tz$ or $x=(t\theta(z))^i (tz)^j t$ and $y=zt$.\\
If $x=(\theta(z)t)^i (zt)^j z$ and $y=tz$, then $x y =(\theta(z)t)^i (zt)^{j+1} z$ and $\theta(y)x =(\theta(z)t)^{i+1} (zt)^j z$. This implies $H(xy, \theta(y)x)=H(z, \theta(z))=1$.\\
If $x=(\theta(z)t)^i \theta(z)$ and $ y= tz$, then $xy = (\theta(z)t)^{i+1} z$ and $\theta(y)x = (\theta(z)t)^{i+1} \theta(z) $. This implies $H(xy, \theta(y)x)=H( z, \theta(z))=1$.\\
If $x=(t\theta(z))^i (tz)^j t$ and $y=zt$, then $xy=(t\theta(z))^i tz (tz)^j t$ and $\theta(y)x= (t \theta(z))^i (t\theta(z)) (tz)^j t$. This implies $H(xy, \theta(y)x)=H(z, \theta(z))=1$. 

\end{proof}

We observe the following from Theorem \ref{mainstrucrure}.
\begin{observation}\label{existenceofxfory}
If a word $y$ can not be expressed in the form $tz$ or $zt$ where $H(z, \theta(z))=1$ and $t \in \textup{P}_\theta$, then there does not exist any word $x$ such that $x$ nearly $\theta$-commutes with $y$, and vice versa.
\end{observation}

Let us study the words $y$ that can be expressed in the form $tz$ or $zt$ as in Observation \ref{existenceofxfory}.

\begin{lemma}\label{vbcxdfer2101}
    If $w=tz$ or $zt$ where $t \in \textup{P}_\theta$ and $H(z, \theta(z))=1$, then $w \notin \textup{P}_\theta$ and $w$ is a primitive word.
\end{lemma}
\begin{proof}
We only prove the case when $w=tz$ as the proof is similar when $w=zt$. 
   Let $w= tz$. 
    We assume by the contrary that $w \in \textup{P}_\theta$.  
    Then, $w=tz=\theta(z)t$. 
    If $t =\lambda$, then $w=z \in\textup{P}_\theta$ which is a contradiction. Let $t \in \Sigma^+$. 
    Then by Lemma \ref{44lkc6+},  $w=tz=\theta(z)t$ gives $\theta(z)=(pq)^j$, $t=(pq)^ip$ where $p, q \in \textup{P}_\theta$, $j \geq 1$, $i \geq 0$. Now, $H(z, \theta(z))=1$ implies that $H((qp)^j, (pq)^j)=1$, which is a contradiction as $(pq)^j \in C((qp)^j)$ and Hamming distance between two conjugates can never be one (Theorem \ref{conalmc5++}). Hence, $w \notin \textup{P}_\theta$. 
    
We now prove the second part that if $w=tz$ then $w$ is a primitive word. Let us assume the contrary that,
    \begin{equation}\label{pqoadfw1}
         w=tz=u^k   \text{ for some nonempty primitive word } u  \text{ and } k \geq 2.
    \end{equation}
   
    We have the following cases:
    \begin{itemize}
        \item Case I : $|t|>|u|$ : Then from Equation (\ref{pqoadfw1}), we have, $t=u^l u_1$ and $z=u_2 u^{k-l-1}$ where $1 \leq l < k, u=u_1 u_2$. Since, $t \in \textup{P}_\theta$ we have $t=u^l u_1=\theta(u_1) \theta(u)^l=\theta(t)$ which implies that  $u_1, u_2\in \textup{P}_\theta$. 
        Then, $\theta(z) = \theta(u)^{k-l-1} \theta(u_2) = \theta(u_2) (\theta(u_1) \theta(u_2))^{k-l-1} = u_2 (u_1 u_2)^{k-l-1} =z$, which implies that  $H(z, \theta(z))= H(z,z)=0$, a contradiction to our assumption that $H(z, \theta(z))=1$.
        

        \item Case II : $|t| = |u|$ :  Then from Equation (\ref{pqoadfw1}), we have $t=u$ and $z=u^{k-1}$. Since $u=t \in \textup{P}_\theta$,  and hence  $z =u^{k-1} \in \textup{P}_\theta$, which implies that  $H(z, \theta(z))= H(z,z)=0$, a contradiction.

        \item Case III : $|t| < |u|$ :  Then from Equation (\ref{pqoadfw1}), we have, $t=u_1$ and $z=u_2 u^{k-1}$ where $u=u_1 u_2$ and $|u_1|=|t|$. Since $t=u_1 \in \textup{P}_\theta$ we get $z = u_2(tu_2)^{k-1}$ and 
    \begin{align*}
    H(z, \theta(z))&= H(u_2 u^{k-1}, \theta(u)^{k-1} \theta(u_2) )\\
    &=  H(u_2 (u_1 u_2)^{k-1}, \theta(u_2) (\theta(u_1) \theta(u_2))^{k-1} ) \\
    &=  H(u_2 (tu_2)^{k-1}, \theta(u_2) (t\theta(u_2))^{k-1}) \\
        &= H(u_2, \theta(u_2) ) + H( (tu_2)^{k-1}, (t \theta(u_2))^{k-1} )\\
        &= H(u_2, \theta(u_2) )+ (k-1) H(t,t) +  (k-1) H(u_2, \theta(u_2) )\\
        &= k H(u_2, \theta(u_2) )
        \end{align*}
        Since $k \geq 2$, $H(z, \theta(z))$ can never be one, which is a contradiction.
        
    \end{itemize}
    Therefore, if $w=tz$, then  $w $ is a primitive word and $w\notin \textup{P}_\theta$. 
\end{proof}

From Theorem \ref{mainstrucrure} and Lemma \ref{vbcxdfer2101}, we have the following observation.
\begin{corollary}\label{wsxcvbgfdr}
    For $x, y \in \Sigma^*$, if  $x$ nearly $\theta$-commutes with $y$, then $y$ is a primitive word.
    \end{corollary}
\begin{proof}
    Let $x$ and $y$ be such that $x$ nearly $\theta$-commutes with $y$. Then by Theorem \ref{mainstrucrure}, $y$ is in the form $tz$ or $zt$ where $t$ is a $\theta$-palindrome and $H(z, \theta(z))=1$ and by Lemma \ref{vbcxdfer2101}, $y$ is a primitive word.
\end{proof}

\section{Relation on words that nearly $\theta$-commute}\label{relation}

We now define a relation $R_\theta$ on $\Sigma^*$ as follows: for $x, y \in \Sigma^*$, $x R_\theta y$ means that $x$ nearly $\theta$-commutes with $y$.
In this section, we discuss some properties of $R_\theta$.

 We first show that the relation $R_\theta$ is not reflexive in general. Consider $x=aa$ where $\theta(a)=b$ and $\theta(b) =a$. Since $H(xx, \theta(x)x)=H(aaaa, bbaa)= 2$, $x R_\theta x$ does not hold. We now discuss the necessary and sufficient conditions for a given word $x$ such that $xR_\theta x$. 

\begin{proposition}
    For $x \in \Sigma^+$, $xR_\theta x$ holds if and only if $x$ is in form $\alpha a \theta(\alpha)$ for some $\alpha \in \Sigma^*$, $a \in \Sigma$ with $\theta(a) \neq a$.
\end{proposition}
\begin{proof}
    For $x \in \Sigma^+$, let $xR_\theta x$ holds, i.e., $x$ nearly $\theta$-commutes with $x$. Then $H(xx, \theta(x)x)=1$, i.e., $H(x, \theta(x))=1$. Then by Lemma \ref{lem1c619++}, $x = \alpha a \theta(\alpha)$ for some $\alpha \in \Sigma^*$, $a \in \Sigma$ with $\theta(a) \neq a$.

    Conversely, let  $x = \alpha a \theta(\alpha)$ for some $\alpha \in \Sigma^*$, $a \in \Sigma$ with $\theta(a) \neq a$. Then $H(x, \theta(x))=1$. This implies  $H(xx, \theta(x)x)=1$, i.e., $xR_\theta x$ holds.
\end{proof}

Now, with the help of the following example, we show that $R_\theta$ is not a symmetric relation.
 We also show existence of words $x$ and $y$ such that if $x$ nearly $\theta$-commutes with $y$, then $y$ also nearly $\theta$-commutes with $x$.
\begin{example}\label{exsym1}
    Let $\theta$ be such that $\theta(a)=b$.
    \begin{enumerate}
        \item Consider $x=abb$ and $y=ab abb$. Then, $$H(xy, \theta(y)x)=H(abb ab abb, aab ab abb)=1$$
        and $x$ nearly $\theta$-commutes with $y$. However, $$H(yx, \theta(x)y)=H(ab abb abb, aab ab abb)=3$$ 
        and $y$ does not nearly $\theta$-commute with $x$.
        
        \item Let $x= abb ab aab ab aab$ and $y=ab a ab$. Then, 
           $$ H(xy, \theta(y)x) =H(abb ab aab ab aab ab a ab, ab b ab abb ab aab ab aab )
            =1$$
            $$ H(yx, \theta(x)y)=H(ab a ab abb ab aab ab aab, abb ab abb ab aab ab a ab)=1$$  
       Hence,  $x$ nearly $\theta$-commutes with $y$ and $y$ nearly $\theta$-commutes with $x$. 
    \end{enumerate}
\end{example}

It is evident from Example \ref{exsym1} that if $x$ nearly $\theta$-commutes with $y$
 then $y$ needs not nearly $\theta$-commute with $x$. We now characterize words $x\in \Sigma^+$ and $y \in \Sigma^+$ such that if $xR_\theta y$ holds, then $y R_\theta x$ holds and vice-versa.

Using Lemma \ref{vbcxdfer2101}, we now characterize words $x$ and $y$ such that $xR_\theta y $  and $y R_\theta x$ hold.

\begin{theorem}
    For $x, y \in \Sigma^+$,  $xR_\theta y $  and $y R_\theta x$ hold if and only if one of the following holds:
    \begin{enumerate}
        \item  $x=\theta(y)^i z y^i$ and $y=tz$
    
        \item $x= \theta(y)$ and $y=\theta(z)$
    
        \item    $x=\theta(y)^i t y^{i+1} $ and $y=zt$

    \end{enumerate}
    where  $i \geq 0$, $t \in \textup{P}_\theta$,  $H(z, \theta(z))=1$, and  $H(y, \theta(y))=1$.
\end{theorem}


\begin{proof}
 For $x, y \in \Sigma^+$, let $x$ nearly $\theta$-commutes with $y$ and $y$ nearly $\theta$-commutes with $x$. 
 By Theorem \ref{mainstrucrure}, $x$ nearly $\theta$-commutes with $y$ if and only if one of the following hold.
         \begin{enumerate}
             \item $x=(\theta(z)t)^i (zt)^j z$ and $y=tz$,
             \item $x=(zt)^i z$ and $ y= t\theta(z)$,
             \item $x=(t\theta(z))^i (tz)^j t$ and $y=zt$
         \end{enumerate}
  
where $i, j \geq 0$, $H(z, \theta(z))=1$, $t \in \textup{P}_\theta$. Now, $y$ nearly $\theta$-commutes with $x$ implies $H(yx, \theta(x)y)=1$. 
Then, we have the following cases:
\begin{enumerate}

\item If $x=(\theta(z)t)^i (zt)^j z$ and $y=tz$ then, 
        $yx=tz (\theta(z)t)^i (zt)^j z$
        and,
         \[\theta(x)y = \theta(z) (t \theta(z))^j (tz)^i tz 
         = \theta(z) t ( \theta(z) t)^j (zt)^i z\]
        Comparing $i$ and $j$, we have the following cases:
        \begin{enumerate}
        \item If $i=j$, then  $H(yx, \theta(x)y)=1$ implies $H(tz, \theta(z)t)=1$, i.e., $H(y, \theta(y))=1$. 
        Thus, $x=(\theta(z)t)^i (zt)^i z$ and $y=tz$ with $H(y, \theta(y))=1$ (proving case 1 of the statement).

        \item If  $i>j$, then  $i=j+k$ for some $k >0$. 
        This implies
        $yx=tz (\theta(z)t)^j (\theta(z)t)^k (zt)^j z$ and $\theta(x)y=\theta(z) t ( \theta(z) t)^j (zt)^k (zt)^j z$. 
        Since $k>0$, $H(z, \theta(z))=1$, and $H(yx, \theta(x)y)=1$, we have $k=1$ and $tz=\theta(z)t$, which implies that $y =tz \in \textup{P}_\theta$,  which is a contradiction by Lemma \ref{vbcxdfer2101}. 
        
        \item If  $i < j$, then $j=i+l$ for some $l>0$. This implies $yx=tz (\theta(z)t)^i (zt)^l (zt)^i z$ and $\theta(x)y=\theta(z) t ( \theta(z) t)^i ( \theta(z) t)^l  (zt)^i z$.  
        Since $l>0$, $H(z, \theta(z))=1$, and $H(yx, \theta(x)y)=1$, we have $l=1$ and $tz=\theta(z)t$, which implies that $y =tz \in \textup{P}_\theta$, which is a contradiction by Lemma \ref{vbcxdfer2101}.
        \end{enumerate}
    
\item If  $x=(zt)^i z$ and $ y= t\theta(z)$ then, 
        $yx= t \theta(z) (zt)^i z $ and 
        $\theta(x)y = \theta(z) (t \theta(z))^i t \theta(z) = \theta(z) t ( \theta(z) t)^i  \theta(z).$
         Now, $$H(yx, \theta(x)y) = H(t \theta(z), \theta(z) t) + H((zt)^i,  ( \theta(z) t)^i) + H(z, \theta(z)).$$
         Since $H(z, \theta(z))=1$ and $H(yx, \theta(x)y)=1$, we have, $ H(t \theta(z), \theta(z) t)=0$ and $H((zt)^i,  ( \theta(z) t)^i)=0$. 
         Now, $H(t \theta(z), \theta(z) t)=0$ implies $t\theta(z) = \theta(z) t$. If $t \in \Sigma^+$, then by Lemma \ref{44lkc6+}, $\theta(z)$ is a $\theta$-palindrome, which is a contradiction. Thus, $t=\lambda$.
         Since $H(z, \theta(z))=1$, $H((zt)^i,  ( \theta(z) t)^i)=0$ implies $i=0$.
         Thus, $x=z$ and $y=\theta(z)$. Here note that $H(y, \theta(y))=1$ (proving case 2 of the statement).

 \item If  $x=(t\theta(z))^i (tz)^j t$ and $y=zt$, then $yx = z t                       (t\theta(z))^i (tz)^j t$ and $\theta(x)y = t (\theta(z)t)^j (zt)^i zt = t (\theta(z)t)^j z (tz)^i t$
         and 
        $$H(yx, \theta(x)y) =H(z t (t\theta(z))^i (tz)^j t, t (\theta(z)t)^j z (tz)^i t). $$

        Based on values of $i$ and $j$, we now have the following cases: 
         
        \begin{enumerate}
        
                \item If $i=j$  then,  
             \begin{align*}
                 H(yx, \theta(x)y)&=H(z t(t\theta(z) )^{i}, t (\theta(z)t)^iz)\\
                 &=H(z tt(\theta(z)t)^{i-1}, t\theta(z)t (\theta(z)t)^{i-1}) +H(\theta(z), z)
             \end{align*}
             If $i=0$ then, $H(yx, \theta(x)y) = H(zt, tz)$. Since $tz \in C(zt)$,
            by Theorem \ref{conalmc5++} we have $H(yx, \theta(x)y) \neq 1$, which is a contradiction. If $i > 0$ then, 
            since $H(yx,\theta(x)y) =1$ and $ H(\theta(z), z)=1$, we have $z t = t\theta(z)$, which implies that $y=zt \in \textup{P}_\theta$, which is a contradiction by Lemma \ref{vbcxdfer2101}.

            \item   If $i>j$,  then  $i= j+m_1$ for some $m_1 >0$. We now have two cases                based on values of $j$. 
\begin{itemize}
    \item If $j = 0$, then $i=m_1$,  $yx = z t (t\theta(z))^i  t$ and $\theta(x)y = t z (tz)^i t$.
              Since $i>0$, $H(yx, \theta(x)y)=1$ and $ H(z, \theta(z))=1$, we have $i=1$ and $tz=zt$. Since $t \in \textup{P}_\theta$ and $z \notin \textup{P}_\theta$, $tz=zt$ implies $t = \lambda$. Thus, 
                $x=\theta(z)$ and $y=z$. Here note that $H(y, \theta(y))=1$ (proving case 2 of the statement).

                \item If $j>0$,  then
                $yx = z t (t\theta(z))^j (t\theta(z))^{m_1} (tz)^j t$  and $\theta(x)y = t (\theta(z)t)^j z( t z)^{m_1} (tz)^j t$.
                This implies 
                \begin{align*}
                    H(yx, \theta(x)y) &  = H(z t (t\theta(z))^j, t (\theta(z)t)^j z) + H((t\theta(z))^{m_1} (tz)^j t, ( t z)^{m_1} (tz)^j t)\\
                    & = H(z t (t\theta(z))^j, t (\theta(z)t)^j z) + m_1
                \end{align*}
                Since $m_1>0$ and $H(yx, \theta(x)y)=1$, we have $m_1=1$ and $z t (t\theta(z))^j = t (\theta(z)t)^j z$. 
                Then as $j>0$, we have $\theta(z)=z$, which is a contradiction.
\end{itemize}

 \item  If $i< j$ then $j=i+m$ for some $m>0$. We now have the following cases based on values of $i$. 

        \begin{itemize}
        \item If $i=0$, then $j=m>0$, $yx = z t t ( z t)^{j-1} zt $, $\theta(x)y = t (\theta(z)t)^j z t =  t \theta(z) t (\theta(z)t)^{j-1} z t$. Then $H(yx, \theta(x)y) =1$ implies 
                \begin{align}
                &H(zt, t\theta(z)) + H((zt)^{j-1}, (\theta(z)t)^{j-1}) = 1 \nonumber \\
                \implies &  H(zt, t\theta(z)) + j-1 = 1 \nonumber \\
                \implies & j = 2-  H(zt, t\theta(z)) \label{3rd3}
                \end{align}
                If $H(zt, t\theta(z))=0$, i.e., $zt = t\theta(z)$, then $zt \in \textup{P}_\theta$, which is a contradiction by Lemma \ref{vbcxdfer2101}. Thus, $H(zt, t\theta(z)) \geq 1$.
                Since $j > 0$, $H(zt, t\theta(z))\leq 1$ which implies that, $H(zt, t\theta(z))=1$ and from Equation (\ref{3rd3}) we have,  $j=1$. Hence, $x= tz t$ and $y=zt$ with $H(y, \theta(y))=1$ (proving a particular case of case 3 of the statement).

               \item  If $i>0$,            
                then $yx = z t (t\theta(z))^i (tz)^{m-1} (tz)^{i+1} t$ and $$\theta(x)y = t (\theta(z)t)^i \theta(z)( t \theta(z))^{m-1}    (tz)^{i+1} t.$$
                Then, $H(yx, \theta(x)y)=1$ implies that
                \begin{equation}\label{3rd4th}
                     H(zt(t\theta(z))^{i}, t (\theta(z)t)^i \theta(z)) + H((tz)^{m-1},( t \theta(z))^{m-1} )=1.
                \end{equation}
                If $m=1$, then Equation (\ref{3rd4th}) implies 
                \begin{align*}
                     &H(zt(t\theta(z))^{i}, t (\theta(z)t)^i \theta(z)) =1 \\
                    \implies & H(zt(t\theta(z))^{i-1} t \theta(z), t (\theta(z)t)^i \theta(z)) =1\\
                    \implies & H(zt(t\theta(z))^{i-1} t , t (\theta(z)t)^{i-1} \theta(z)t) =1\\
                    \implies & H(zt t (\theta(z)t)^{i-1} , t \theta(z)t (\theta(z)t)^{i-1}) =1\\
                    \implies & H(zt t  , t \theta(z)t) =1\\
                    \implies & H(zt  , t \theta(z)) =1\\
                    \implies & H(y, \theta(y))=1.
                \end{align*}
                Thus,  $x=(t\theta(z))^i (tz)^{i+1} t$ and $y=zt$ with $H(y, \theta(y))=1$ (proving case 3 of the statement). 
        
                If $m \geq 2$, then Equation (\ref{3rd4th}) implies that  
                \begin{align*}
                     &H(zt(t\theta(z))^{i}, t (\theta(z)t)^i \theta(z)) + (m-1) H(tz, t \theta(z) )=1\\
                     \implies & H(zt(t\theta(z))^{i}, t (\theta(z)t)^i \theta(z)) + (m-1)H(z,\theta(z))=1
                     \end{align*}
                     Since $H(yx, \theta(x)y)=H(z,\theta(z)) =1$ and $m \geq 2$, we have,
                     $$H(zt(t\theta(z))^{i}, t (\theta(z)t)^i \theta(z)) =0,$$
                i.e., $zt(t\theta(z))^{i} = t (\theta(z)t)^i \theta(z)$. This implies $zt  = t \theta(z) $, i.e., $y=zt \in \textup{P}_\theta$, which is a contradiction by Lemma \ref{vbcxdfer2101}.
                \end{itemize}
                \end{enumerate}

\end{enumerate}

Converse can be easily verified. 
\end{proof}


The next example shows that the relation $R_\theta$ is not transitive. 
\begin{example}
   Let $\theta$ be such that $\theta(a)=b$.
    Consider $u=abaab$, $v=aab$, and $w= baaab$. Since $H(uv, \theta(v)u) = H(vw, \theta(w)v)=1$, $u R_\theta v$ and $v R_\theta w$ both holds. But as $H(uw, \theta(w)u)>1$, $u R_\theta w$ does not hold.
\end{example}

We now impose the following.
$$\textup{Pr}_\theta(w) =\{z\in \Sigma^+~:~ w= zx \text{ for some } x \in \Sigma^*, ~H(z,\theta(z)) =1\}$$
$$\textup{Sf}_\theta(w) =\{z\in \Sigma^+~:~ w= xz \text{ for some } x \in \Sigma^*, ~H(z,\theta(z)) =1\}.$$
 We now impose some conditions on $u, v$ and show that $u R_\theta v$ and $v R_\theta w$ implies $u R_\theta w$.
\begin{proposition}
    Let $u,v, w \in \Sigma^+$ such that $|u| < |v|$,  $|\textup{Sf}_\theta(v)|\le 1$ and $|\textup{Pr}_\theta(v)|\le 1$.
    If $u$ nearly $\theta$-commutes with $v$ and $v$ nearly $\theta$-commutes with $w$, then $u$ nearly $\theta$-commutes with $w$.
\end{proposition}
\begin{proof}
    Let $u$ nearly $\theta$-commute with $v$ and $v$ nearly $\theta$-commute with $w$. Then as $0<|u| < |v|$, by Theorem \ref{mainstrucrure}, one of the following cases hold:
    \begin{enumerate}
        \item $u= z$ and $v=tz$.
        \item $u=z$ and $v=t \theta(z)$.
        \item $u=t$ and $v=zt$  
   \end{enumerate}
    for some $t \in \Sigma^+ \cap \textup{P}_\theta$ and $z \in \Sigma^+$ with $H(z, \theta(z))=1$.

    \begin{itemize}
        \item \textit{Case 1}: Let  $u= z$ and $v=tz$. Since $v$ nearly $\theta$-commutes with $w$, by Theorem \ref{mainstrucrure} one of the following hold true.
        \begin{enumerate}
            \item $v = (\theta(z_1)t_1)^i (z_1t_1)^j z_1$ and $w=t_1 z_1$
            \item $v=(z_1t_1)^i z_1$ and $w= t_1 \theta(z_1)$ \item $v=(t_1\theta(z_1))^i (t_1z_1)^j t_1$ and $w=z_1t_1$
        \end{enumerate} 
        where $i, j \geq 0$, $H(z_1, \theta(z_1))=1$, $t_1 \in \textup{P}_\theta$.
        
        \begin{itemize}
            \item \textit{Case 1.1}: If $v=tz = (\theta(z_1)t_1)^i (z_1t_1)^j z_1$ and $w=t_1 z_1$, with $z, z_1 \in \textup{Sf}_\theta(v)$, then as $|\textup{Sf}_\theta(v)|\leq 1$,
             we have $z = z_1$ which implies that $t=(\theta(z)t_1)^i (zt_1)^j$. 
             Since $t \in \Sigma^+$, $i+j>0$.
            As $H(z,\theta(z)) =1$ and $t,t_1\in \textup{P}_\theta$, we have by Lemma \ref{vbcxdfer2101}, $t_1z, zt_1, t_1\theta(z), \theta(z)t_1   \notin \textup{P}_\theta$.  
            Note that $t, t_1\in \textup{P}_\theta$ and hence, $t = (\theta(z)t_1)^i (zt_1)^j = (t_1\theta(z))^j (t_1z)^i$ implies that 
            either $\theta(z)t_1 = t_1\theta(z) $ 
            or $\theta(z)t_1 = t_1z$ 
            or $zt_1 =t_1\theta(z)$. 
            Now, $\theta(z)t_1 = t_1z$ or  $zt_1 =t_1\theta(z)$ are not possible as they imply that $\theta(z)t_1 \in \textup{P}_\theta$ or $zt_1 \in \textup{P}_\theta$. 
            If $\theta(z)t_1 = t_1\theta(z)$, and $t_1 \in \Sigma^+$, then by Lemma \ref{44lkc6+}, $t_1$ and $z$ are powers of a common word which implies that $z \in \textup{P}_\theta$, a contradiction. Hence, $t_1=\lambda$ and  
            $w=z$.
            This implies as $u=z$, $$H(uw,\theta(w)u) = H(zz,\theta(z)z) = H(z,\theta(z))=1.$$
           Thus, $u$ nearly $\theta$-commutes with $w$.

            \item \textit{Case 1.2}: If  $v=tz=(z_1t_1)^i z_1$ and $w= t_1 \theta(z_1)$, then as $z, z_1 \in \textup{Sf}_\theta(v)$ and $|\textup{Sf}_\theta(v)|\leq 1$, we have $z =z_1$ and  $t= (z t_1)^i$. 
            Since $t \in \Sigma^+$, $i >0$. Now, $t, t_1 \in \textup{P}_\theta$ implies $t= (z t_1)^i = (t_1 \theta(z))^i$. This impies $zt_1 = t_1 \theta(z)$, i.e., $zt_1 \in \textup{P}_\theta$, which is a contradiction by Lemma \ref{vbcxdfer2101}.
            Thus,  $v=tz=(z_1t_1)^i z_1$ is not possible.

            \item \textit{Case 1.3}:  Let $v=tz=(t_1\theta(z_1))^i (t_1z_1)^j t_1$ and $w=z_1t_1$. 
            
            If $t_1 \in \Sigma^+$, then as $z \in \textup{Sf}_\theta(v)$ and $|\textup{Sf}_\theta(v)|\leq 1$,
             either $z=t_1z_1t_1$  and $t=(t_1 \theta(z_1))^i (t_1z_1)^{j-1}$
            or $z=t_1 \theta(z_1) t_1$ and $t=(t_1 \theta(z_1))^{i-1}$. Then similar to \textit{Case 1.1}, we have a contradiction for $t=(t_1 \theta(z_1))^i (t_1z_1)^{j-1}$ and $t=(t_1 \theta(z_1))^{i-1}$.

            If $t_1 =\lambda$, then $z=z_1$ and $t= \theta(z_1)^{i'} z_1^{i'} $ for some $i' \geq 1$. Then, $w=z_1t_1 = z $ and 
            $$H(uw,\theta(w)u) = H(zz,\theta(z)z) = H(z,\theta(z))=1.$$
           Thus, $u$ nearly $\theta$-commutes with $w$.
            
        \end{itemize}

\item \textit{Case 2}: The case when $u=z$ and $v=t \theta(z)$ is very similar to Case 1 and we omit the proof.

          \item  \textit{Case 3}: Let $u=t$ and $v=zt$.
        Since $v$ nearly $\theta$-commutes with $w$, by Theorem \ref{mainstrucrure} one of the following holds true: 
        \begin{enumerate}
            \item $v = (\theta(z_1)t_1)^i (z_1t_1)^j z_1$ and $w=t_1 z_1$  
            \item $v=(z_1t_1)^i z_1$ and $w= t_1 \theta(z_1)$ 
            \item $v=(t_1\theta(z_1))^i (t_1z_1)^j t_1$ and $w=z_1t_1$
        \end{enumerate} for some $i, j \geq 0$, $H(z_1, \theta(z_1))=1$, $t_1 \in \textup{P}_\theta$.
        
        \begin{itemize}
            \item \textit{Case 3.1}: If $v=zt = (\theta(z_1)t_1)^i (z_1t_1)^j z_1$ and $w=t_1 z_1$, then
            as $z \in \textup{Pr}_\theta(v)$ and $|\textup{Pr}_\theta(v)|\leq 1$, we have either $z = \theta(z_1)$ and  $t= t_1 (\theta(z_1)t_1)^{i-1} (z_1t_1)^jz_1$ or $z = z_1$ and  $t= (t_1 z_1)^{j}$.  
            As $H(z_1,\theta(z_1)) =1$ and $t,t_1\in \textup{P}_\theta$, we have by Lemma \ref{vbcxdfer2101}, $t_1 z_1, z_1 t_1, t_1\theta(z_1), \theta(z_1)t_1 \notin \textup{P}_\theta$. 
Since $t, t_1 \in \textup{P}_\theta$, $$t= t_1 (\theta(z_1)t_1)^{i-1} (z_1t_1)^jz_1 =  (t_1 \theta(z_1))^{i-1} t_1 z_1 (t_1 z_1)^j =(\theta(z_1)t_1)^j \theta(z_1) t_1 (z_1 t_1)^{i-1}.$$  
This implies that either $t_1\theta(z_1) = \theta(z_1)t_1$ or $t_1 \theta(z_1) = z_1t_1$ or $t_1z_1 = \theta(z_1)t_1$.
            Now $t_1 \theta(z_1) = z_1t_1$ or $t_1z_1 = \theta(z_1)t_1$ are not possible as they imply that $t_1 \theta(z_1) \in \textup{P}_\theta$ or $t_1 z_1 \in \textup{P}_\theta$.
             If $t_1 \theta(z_1) = \theta(z_1) t_1$, and $t_1 \in \Sigma^+$, then by Lemma \ref{44lkc6+}, $t_1$ and $\theta(z_1)$ are powers of a common word which implies that $z_1 \in \textup{P}_\theta$, a contradiction. Thus,  $t_1=\lambda$.
             This implies $t=\theta(z_1)^{j'} z_1^{j'}$ for some $j'\geq 1$ and $w=z_1=\theta(z)$.
             Then, $u= t = \theta(z_1)^{j'} z_1^{j'} = z^{j'} \theta(z)^{j'}$. Thus, $$H(uw, \theta(w)u)=H(z^{j'} \theta(z)^{j'} \theta(z), z z^{j'} \theta(z)^{j'} )=H(z^{j'} \theta(z) \theta(z)^{j'},  z^{j'} z \theta(z)^{j'} )=1.$$
             Therefore, $u$ nearly $\theta$-commutes with $w$.

             Now, as $t, t_1 \in \textup{P}_\theta$, $t= (t_1 z_1)^{j} = (\theta(z_1)t_1)^j$ which implies $t_1 z_1 = \theta(z_1)t_1$, which is a contradiction as $t_1 z_1 \notin \textup{P}_\theta$.
            

           \item \textit{Case 3.2}: If  $v= zt = (z_1t_1)^i z_1$ and $w= t_1 \theta(z_1)$, then 
             as $z, z_1 \in \textup{Pr}_\theta(v)$ and $|\textup{Pr}_\theta(v)|\leq 1$, we have $z = z_1$ and  $t=  (t_1 z_1)^{i}$. 
             Now $t=  (t_1 z_1)^{i}$ gives us a contradiction as $t_1 z_1 \notin \textup{P}_\theta$.

 \item \textit{Case 3.3}: Let $v= z t =(t_1\theta(z_1))^i (t_1z_1)^j t_1$ and $w=z_1t_1$. If $t_1 \in \Sigma^+$, then
             as $z \in \textup{Pr}_\theta(v)$ and $|\textup{Pr}_\theta(v)|\leq 1$, we have
            either $z=t_1 \theta(z_1) t_1$ and $t= \theta(z_1) (t_1 \theta(z_1))^{i-2} (t_1z_1)^{j} t_1$
            or   $z=t_1 \theta(z_1) t_1$ and $t= (z_1 t_1)^{j}$      
            or $z=t_1 z_1 t_1$ and $t= (z_1t_1)^{j-1}$.

            Since $t, t_1 \in P_\theta$, $t= \theta(z_1) (t_1 \theta(z_1))^{i-2} (t_1z_1)^{j} t_1 = (\theta(z_1) t_1)^{i-2} \theta(z_1) t_1 (z_1 t_1)^{j} = (t_1 \theta(z_1))^j t_1z_1 (t_1 z_1)^{i-2}$. This implies either $\theta(z_1)t_1 = t_1 \theta(z_1)$ or $\theta(z_1)t_1 = t_1 z_1$ or $z_1 t_1 = t_1 \theta(z_1)$.
             Now $\theta(z_1)t_1=t_1z_1$ or $t_1 \theta(z_1) = z_1t_1$  are not possible as they imply that $t_1 z_1 \in \textup{P}_\theta$ or $t_1 \theta(z_1) \in \textup{P}_\theta$. If $\theta(z_1)t_1 = t_1 \theta(z_1)$, then by Lemma \ref{44lkc6+}, $t_1$ and $\theta(z_1)$ are powers of a common word which implies that $z_1 \in \textup{P}_\theta$, a contradiction.

             Now  $t= (z_1 t_1)^{j}$ or $t= (z_1t_1)^{j-1}$ are not possible as $z_1 t_1 \notin \textup{P}_\theta$.

            
            If $t_1 =\lambda$, then $z = \theta(z_1)$ and $t= \theta(z_1)^{i'} z_1^{i'} $ for some $i'\geq 1$ and $w=z_1 t_1 = \theta(z) $ and $u= (\theta(z_1))^{i'} z_1^{i'} = z^{i'} \theta(z)^{i'} $.
            Thus, $$H(uw, \theta(w)u)=H(z^{i'} \theta(z)^{i'} \theta(z), z z^{i'} \theta(z)^{i'} )=H(z^{i'} \theta(z) \theta(z)^{i'},  z^{i'} z \theta(z)^{i'} )=1.$$
             Therefore, $u$ nearly $\theta$-commutes with $w$.
              \end{itemize}
    \end{itemize}
\end{proof}

\section{Properties of the set of all $x$ such that $x$ nearly $\theta$-commutes with a given $y$}\label{Property}

In this section, for a given non-empty word $y\in \Sigma^+$, we define a set denoted by $L_\theta(y)$, which consists of all words that nearly $\theta$-commutes with the given word $y$ and discuss few combinatorial properties of $L_\theta(y)$.
We first show that $L_\theta(y)$ is regular (Theorem \ref{thcf}). We then characterize words $y_1$ and $y_2$ such that equality of $L_\theta(y_1)$ and $L_\theta(y_2)$ holds (Theorem \ref{theq}).

For a given $y \in \Sigma^+$, we define, $$L_\theta(y) = \{ x : H(xy, \theta(y)x) = 1\}.$$

\begin{example}\label{exmplea-1}
       Let $y=baba a$ where $\theta(a)=b$. 
    Then $y$ can not be written in the form $zt$ but it can be written in the form $tz$ where $t \in \textup{P}_\theta$ and $H(z, \theta(z))=1$. 
    In fact, $y$ has exactly two distinct representations in the form $tz$, i.e., $y=t_1z_1 = t_2z_2$ where $t_1=baba$, $z_1=a$, $t_2=ba$ and $z_2=baa$ with $t_1, t_2 \in \textup{P}_\theta$ and $H(z_1, \theta(z_1))=H(z_2, \theta(z_2))=1$.
    Then, $L_\theta(y) = \{ (b baba)^{i} (a baba)^{j} a: i, j \geq 0 \} \cup \{ (b baba)^{i} b: i\geq 0  \} \cup \{ (bba ba)^{i} (baa ba)^{j} baa : i, j \geq 0  \} \cup \{ (bba ba)^{i} bba: i\geq 0    \}$.
\end{example}


We first show that for any given non-empty word $y$, the set of all words that nearly $\theta$-commutes with $y$ is a regular language. 

\begin{theorem}\label{thcf}
    For a given $y \in \Sigma^+$, $L_\theta(y)$ is a regular language.
\end{theorem}
\begin{proof}
If $y$ can not be expressed in the form $tz$ or $zt$ where $H(z, \theta(z))=1$ and $t \in \textup{P}_\theta$, then by Theorem \ref{mainstrucrure}, $L_\theta(y) = \emptyset$ which is a regular language.

Now suppose that $y$ can be expressed in the form $tz$ or $zt$, where $H(z, \theta(z))=1$ and $t \in \textup{P}_\theta$. 
Since $|y|$ is finite, $y$ admits only finitely many representations of these forms. Assume that there are exactly $k \geq 0$ many distinct representations of $y$ in the form $tz$ and $m \geq 0$ many distinct representations of $y$ in the form $zt$ with $m+k>0$. Then, we can write $y=t_1z_1=t_2z_2 = \cdots = t_k z_k = z_1't_1'= z_2't_2' = \cdots =z_m't_m'$ where $H(z_i, \theta(z_i))=H(z_j', \theta(z_j'))=1$ and $t_i, t_j' \in \textup{P}_\theta$ for all $1 \leq i \leq k$ and $1 \leq j \leq m$.
Then by Theorem \ref{mainstrucrure},
\[L_\theta(y) = L_1 \cup L_2 \cup L_3 \] where 
\[ L_1 = \{ (\theta(z_i)t_i)^{l_1} (z_it_i)^{l_2} z_i ~|~ 1 \leq i \leq k \text{ and } l_1, l_2 \geq 0\}, \] 
\[L_2 = \{ (\theta(z_i)t_i)^{l_1} \theta(z_i) ~|~ 1 \leq i \leq k \text{ and } l_1 \geq 0\}, \] 
\[ L_3 =  \{ (t_j' \theta(z_j'))^{l_1} (t_j'z_j')^{l_2} t_j' ~|~ 1 \leq j \leq m \text{ and } l_1, l_2 \geq 0\}.\]

We now define a regular grammar $G = (N, \Sigma, P, S)$ where $N = \{S, A_i, B_i, C_j, D_j ~|~ 1 \leq i \leq k, 1 \leq j \leq m \}$ is the set of non-terminals, $\Sigma$ is the set of terminals, $S$ is the start symbol and $P$ is the set of derivation rules which contains the following rules:
\begin{itemize}
    \item $ S \rightarrow A_i ~|~ C_j $
    \item $A_i \rightarrow \theta(z_i) t_i A_i ~|~ \theta(z_i) ~|~ B_i $
    \item $B_i \rightarrow z_i t_i B_i ~|~ z_i $
    \item $C_j \rightarrow t_j' \theta(z_j') C_j ~|~ D_j  $
   \item $D_j \rightarrow t_j' z_j' D_j ~|~ t_j' $
\end{itemize}
for all $1 \leq i \leq k$ and $1 \leq j \leq m$.
We now show that $L_\theta(y)=L(G)$ where $L(G)$ is the language generated by the grammar $G$.

Let $w \in L_\theta(y)$. Then,  $w= (\theta(z_i)t_i)^{l_1} (z_it_i)^{l_2} z_i$ or $w= (\theta(z_i)t_i)^{l_1} \theta(z_i)$ or $w=(t_j' \theta(z_j'))^{l_1} (t_j'z_j')^{l_2} t_j'$ for some $1 \leq i \leq k$, $1 \leq j \leq m$ and $l_1, l_2 \geq 0$. 
Since $A_i \rightarrow (\theta(z_i)t_i)^{l_1}$ and $B_i \rightarrow (z_it_i)^{l_2}$, $S \rightarrow (\theta(z_i)t_i)^{l_1} (z_it_i)^{l_2} z_i$, i.e., $(\theta(z_i)t_i)^{l_1} (z_it_i)^{l_2} z_i \in L(G)$.
Similarly, 
$(\theta(z_i)t_i)^{l_1} \theta(z_i) \in L(G)$ and $(t_j' \theta(z_j'))^{l_1} (t_j'z_j')^{l_2} t_j' \in L(G)$. Thus $w \in L(G)$. Therefore,  $L_\theta(y) \subseteq L(G)$.

Let $u \in L(G)$. Then either 
$u=(\theta(z_i)t_i)^{l_1} (z_it_i)^{l_2} z_i$ or $u= (\theta(z_i)t_i)^{l_1} \theta(z_i)$ or $u=(t_j' \theta(z_j'))^{l_1} (t_j'z_j')^{l_2} t_j'$ for some $1 \leq i \leq k$, $1 \leq j \leq m$ and $l_1, l_2 \geq 0$. Then clearly $u \in L_\theta(y)$, i.e., $L(G) \subseteq L_\theta(y)$.

Therefore $L_\theta(y)=L(G)$. This implies  $L_\theta(y)$ is a regular language.

\end{proof}

We now discuss the structures of words $y_1$ and $y_2$ such that  $L_\theta(y_1) = L_\theta(y_2)$.
From Observation \ref{existenceofxfory},  $L_\theta(y_1) = L_\theta(y_2) =\emptyset$ if and only if $y_1$ and  $y_2$ can not be expressed in the form $tz$ or $zt$ where $H(z, \theta(z))=1$ and $t \in \textup{P}_\theta$. We now characterize $y_1$ and  $y_2$ such that  $L_\theta(y_1) = L_\theta(y_2)$, and both $L_\theta(y_1) $ and $ L_\theta(y_2)$ are non-empty.

\begin{theorem}\label{theq}
    For $y_1, y_2 \in \Sigma^*$, let $L_\theta(y_1)$ and $L_\theta(y_2)$ be non-empty. Then
    $L_\theta(y_1) = L_\theta(y_2)$ if and only if $y_1 =y_2$.
\end{theorem}
\begin{proof}
    For $y_1, y_2 \in \Sigma^*$, let $L_\theta(y_1)$ and $L_\theta(y_2)$ be non-empty sets with $L_\theta(y_1) = L_\theta(y_2)$. 
    Without loss of generality, let $|y_1| \leq |y_2|$.
    Since  $L_\theta(y_1)$ is non-empty,  there exists $t_1, z_1 \in \Sigma^*$ where $t_1 \in \textup{P}_\theta$ and $H(z_1, \theta(z_1))=1$ such that  $y_1=t_1z_1$  or  $y_1=z_1t_1$.
    If $y_1=t_1z_1$, then for all $i, j \geq 0$, $(\theta(z_1)t_1)^i (z_1t_1)^j z_1 \in L_\theta(y_1)$ and $(\theta(z_1)t_1)^i \theta(z_1) \in L_\theta(y_1)$.
    If $y_1=z_1 t_1$, then for all $i, j \geq 0$, $(t_1\theta(z_1))^i (t_1z_1)^j t_1 \in L_\theta(y_1)$.
    We now have the following cases:
    \begin{itemize}
        \item \textbf{Case 1}: Let $y_1=t_1z_1$ and  $(\theta(z_1)t_1)^i (z_1t_1)^j z_1 \in L_\theta(y_1)$ for all $i, j \geq 0$.
        
        Then $z_1 t_1 z_1 \in L_\theta(y_1)$, taking $i=0$ and $ j=1$. This implies $z_1 t_1 z_1 \in L_\theta(y_2)$. Then there exist $t_2, z_2 \in \Sigma^*$ with $t_2 \in \textup{P}_\theta$ and $H(z_2, \theta(z_2))=1$  such that one of the following holds:
        \begin{enumerate}
            \item $z_1 t_1 z_1 = (\theta(z_2)t_2)^{i'} (z_2t_2)^{j'} z_2$ for some $i', j' \geq 0$ where $y_2= t_2 z_2$.
            
            \item  $z_1 t_1 z_1 = (z_2t_2)^{i'} z_2$ for some $i' \geq 0$ where $y_2= t_2 \theta(z_2)$.

            \item $z_1 t_1 z_1 = (t_2 \theta(z_2))^{i'} (t_2z_2)^{j'} t_2$ for some $i', j' \geq 0$ where $y_2= z_2 t_2$.
        \end{enumerate}

   Let us discuss the cases one by one.
        \begin{itemize}
            \item \textbf{Case 1.1}  Let, \begin{equation}\label{equty0}
            z_1 t_1 z_1 = (\theta(z_2)t_2)^{i'} (z_2t_2)^{j'} z_2 \text{ for some } i', j' \geq 0.
        \end{equation}
        \begin{itemize}
            \item 
        If $|y_1|<|y_2|$, i.e., $|t_1 z_1| < |t_2 z_2|$, then $|z_1t_1z_1|<2|z_2t_2|$ and thus $i'+j'<2$. Therefore, from Equation (\ref{equty0}), we have either $z_1t_1z_1=z_2$ or $z_1 t_1 z_1 = z_2 t_2 z_2$ or $z_1 t_1 z_1 = \theta(z_2)t_2 z_2$.
        Let  $z_1t_1z_1=z_2$. Then $H(z_2, \theta(z_2))=H(z_1 t_1 z_1, \theta(z_1) t_1 \theta(z_1))=2$ as $H(z_1, \theta(z_1))=1$, which is a contradiction as $H(z_2, \theta(z_2))=1$.
        Let $z_1 t_1 z_1 = z_2 t_2 z_2$.
        Then as $|t_1 z_1| < |t_2 z_2|$, $|z_1|> |z_2|$. This implies $z_1 = z_2 \alpha=\alpha' z_2$ for some $\alpha, \alpha' \in \Sigma^+$. Then by Lemma \ref{44lkc6+}, 
        $\alpha' = (pq)^{l'}$, $z_2 = (pq)^{l}p$ and $\alpha=(qp)^{l'}$ for some $q \in \Sigma^*$, $p \in \Sigma^+$, $l'\geq 1$ and $l \geq 0$.
        Now $z_1 = z_2 \alpha=\alpha' z_2$ and $z_1 t_1 z_1 = z_2 t_2 z_2$ implies $\alpha t_1 \alpha'=t_2$ which implies that $\alpha = \theta(\alpha')$. Then using $\alpha' = (pq)^{l'}$ and $\alpha=(qp)^{l'}$, we have $p, q \in \textup{P}_\theta$ and hence, $z_2 \in \textup{P}_\theta$, a contradiction.
        Similarly, when $z_1 t_1 z_1 = \theta(z_2)t_2 z_2$, we have a contradiction.

        \item If $|y_1|=|y_2|$, then from Equation (\ref{equty0}),  $z_1 t_1 z_1 = z_2 t_2 z_2$ or $z_1 t_1 z_1 = \theta(z_2)t_2 z_2$. Now, $z_1 t_1 z_1 = z_2 t_2 z_2$  implies that $y_1=y_2$ and $z_1 t_1 z_1 = \theta(z_2)t_2 z_2$ implies that $z_1=\theta(z_2) = z_2$ which is a contradiction.
        \end{itemize}

         \item \textbf{Case 1.2} Let,  \begin{equation}\label{equty1}
             z_1 t_1 z_1 = (z_2t_2)^{i'} z_2  \text{ for some } i' \geq 0.
        \end{equation} 
        \begin{itemize}
            \item If $|y_1|<|y_2|$, i.e., $|t_1 z_1| < |t_2 z_2|$, then $|z_1t_1z_1|<2|z_2t_2|$. Then from Equation (\ref{equty1}), we have either $z_1t_1z_1=z_2$ or $z_1 t_1 z_1 = z_2 t_2 z_2$. Similar to Case 1.1, we have a contradiction.
        \item If $|y_1|=|y_2|$, then from Equation (\ref{equty1}),  $z_1 t_1 z_1 = z_2 t_2 z_2$. This gives $z_1=z_2$ and $t_1=t_2$. Then, $y_1=t_1z_1$ and $y_2=t_1 \theta(z_1)$ which implies that, $\theta(z_1)t_1 z_1t_1 z_1 \in L_\theta(y_1) \setminus L_\theta(y_2)$, a contradiction. 
        \end{itemize}

        \item \textbf{Case 1.3}  Let, \begin{equation}\label{equty2}
            z_1 t_1 z_1 = (t_2 \theta(z_2))^{i'} (t_2z_2)^{j'} t_2 \text{ for some } i', j' \geq 0.
        \end{equation}
        \begin{itemize}
            \item 
        If $|y_1| < |y_2|$, then $|z_1t_1z_1|<2|z_2t_2|$. Then from Equation (\ref{equty2}) we have either $z_1t_1z_1=t_2$ or $z_1 t_1 z_1 = t_2 z_2 t_2$ or $z_1 t_1 z_1 = t_2 \theta(z_2) t_2$.
        If $z_1t_1z_1=t_2$, then $\theta(z_1)=z_1$, contradiction.
        Let $z_1 t_1 z_1 = t_2 z_2 t_2$. Then $H(t_2 z_2 t_2, t_2 \theta(z_2) t_2) = H(z_1 t_1 z_1, \theta(z_1) t_1 \theta(z_1))=2$ which is a contradiction as $H(t_2 z_2 t_2, t_2 \theta(z_2) t_2)=1$. Similarly, $z_1 t_1 z_1 = t_2 \theta(z_2) t_2$ is not possible.
        

        \item If $|y_1| = |y_2|$, then from Equation (\ref{equty2}), we have  either $z_1 t_1 z_1 = t_2 z_2 t_2$ or $z_1 t_1 z_1 = t_2 \theta(z_2) t_2$.
        Similar to the above, $z_1 t_1 z_1 = t_2 z_2 t_2$ or $z_1 t_1 z_1 = t_2 \theta(z_2) t_2$ is not possible.
        \end{itemize}

        \end{itemize}

        \item  \textbf{Case 2}: Let , $y_1=t_1z_1$ and $(\theta(z_1)t_1)^i \theta(z_1) \in L_\theta(y_1)$ for all $i \geq 0$. 
        
        Then, $\theta(z_1) t_1 \theta(z_1) \in L_\theta(y_1)$. This implies $\theta(z_1) t_1 \theta(z_1) \in L_\theta(y_2)$.
        Then similar to Case $1$, we show that $y_1=y_2$.

\item  \textbf{Case 3}: Let $y_1=z_1 t_1$ and $(t_1\theta(z_1))^i (t_1z_1)^j t_1 \in L_\theta(y_1)$ for all $i, j \geq 0$.

        Then, $ t_1z_1 t_1 z_1 t_1 \in L_\theta(y_1)$, i.e., $ t_1z_1 t_1 z_1 t_1 \in L_\theta(y_2)$ and there exist $t_2, z_2 \in \Sigma^*$ with $t_2 \in \textup{P}_\theta$ and $H(z_2, \theta(z_2))=1$  such that one of the following holds:
        \begin{enumerate}
            \item $t_1z_1 t_1 z_1 t_1 = (\theta(z_2)t_2)^{i'} (z_2t_2)^{j'} z_2$ for some $i', j' \geq 0$ where $y_2= t_2 z_2$.

            \item  $t_1z_1 t_1 z_1 t_1 = (z_2t_2)^{i'} z_2$ for some $i' \geq 0$ where $y_2= t_2 \theta(z_2)$.

            \item $t_1z_1 t_1 z_1 t_1 = (t_2 \theta(z_2))^{i'} (t_2z_2)^{j'} t_2$ for some $i', j' \geq 0$ where $y_2= z_2 t_2$.
        \end{enumerate}

        \begin{itemize}

         \item \textbf{Case 3.1} Let, $t_1z_1 t_1 z_1 t_1 = (\theta(z_2)t_2)^{i'} (z_2t_2)^{j'} z_2$ for some $i', j' \geq 0$.
         \begin{itemize}
             \item 
        If $|y_1|< |y_2|$, then $t_1z_1 t_1 z_1 t_1 = z_2$ or $t_1z_1 t_1 z_1 t_1 = z_2t_2 z_2$ or 
        $t_1z_1 t_1 z_1 t_1 = \theta(z_2)t_2 z_2$ or $t_1z_1 t_1 z_1 t_1 = \theta(z_2)t_2 z_2t_2 z_2$ or $t_1z_1 t_1 z_1 t_1 = z_2 t_2 z_2t_2 z_2$ or $t_1z_1 t_1 z_1 t_1 = \theta(z_2)t_2 \theta(z_2) t_2 z_2$.
        
        Let $t_1z_1 t_1 z_1 t_1 = z_2$. Then $H(z_2, \theta(z_2)) = H(t_1z_1 t_1 z_1 t_1, t_1 \theta(z_1) t_1 \theta(z_1) t_1)=2$, which is a contradiction as $H(z_2, \theta(z_2))=1$.
        Let $t_1z_1 t_1 z_1 t_1 = \theta(z_2)t_2 \theta(z_2) t_2 z_2$. Then $H(\theta(z_2)t_2 \theta(z_2) t_2 z_2, \theta(z_2)t_2 z_2 t_2 z_2)=1$ which is a contradiction as\\
        $H(t_1z_1 t_1 z_1 t_1, t_1 \theta(z_1) t_1 \theta(z_1) t_1)=2$.
        Similarly, we show that $t_1z_1 t_1 z_1 t_1 = z_2 t_2 z_2t_2 z_2$ and $t_1z_1 t_1 z_1 t_1 = \theta(z_2)t_2 z_2t_2 z_2$ are not possible.

        
        If $t_1z_1 t_1 z_1 t_1 = z_2t_2 z_2$, then we have the following cases:
        \begin{itemize}
            \item If $|z_2| \leq |z_1t_1|$, then $z_1t_1=\alpha' z_2 $ and $t_1z_1 = z_2 \alpha$ for some $\alpha, \alpha' \in \Sigma^*$. Substituting these values in $t_1z_1 t_1 z_1 t_1 = z_2t_2 z_2$, we get $\alpha' =\theta(\alpha)$. Now, $z_1t_1 \alpha = \alpha' z_2 \alpha = \theta(\alpha) z_2 \alpha $ and $\theta(\alpha) t_1 z_1 =\theta(\alpha) z_2 \alpha$. This implies that $z_1t_1 \alpha = \theta(\alpha) t_1 z_1$.
            Using Proposition \ref{ch4defctw32+2}, we have $z_1 \in \textup{P}_\theta$, a contradiction.
            
            \item  If $|z_2| > |z_1t_1|$, then $t_1z_1 t_1 z_1 t_1$ has two periods, $|t_1z_1|$ and $|z_2t_2|$, and $|t_1z_1 t_1 z_1 t_1| \geq |t_1z_1|+|t_2z_2| - gcd(|t_1z_1|, |t_2z_2|)$. Then
            by Fine and Wilf theorem \ref{FineandWilf}, $z_2t_2 \notin Q$, contradiction.
        \end{itemize}
        If $t_1z_1 t_1 z_1 t_1 = \theta(z_2)t_2 z_2$, then we also have two cases:
        \begin{itemize}
            \item If  $|z_2| \leq |z_1t_1|$ then, $z_1t_1 =\alpha' z_2$ and $t_1z_1 = \theta(z_2)\alpha$ for some $\alpha, \alpha' \in \Sigma^*$. Substituting these values in $t_1z_1 t_1 z_1 t_1 = \theta(z_2)t_2 z_2$, we get $\alpha' = \theta(\alpha)$. 
            Then, $z_1t_1 =\theta(\alpha) z_2$ and $t_1z_1 = \theta(z_2)\alpha = \theta(\theta(\alpha)z_2) = \theta(z_1t_1)= t_1 \theta(z_1)$. This implies $z_1 \in \textup{P}_\theta$, contradiction.
            
            \item If $|z_2| > |z_1t_1|$ then, $z_2 = \alpha' z_1 t_1 $ and $\theta(z_2) = t_1 z_1 \alpha$ for some $\alpha, \alpha' \in \Sigma^+$, i.e., 
            $\theta(\alpha' z_1 t_1) = t_1 z_1 \alpha$. This implies 
            $z_1 \in \textup{P}_\theta$, a contradiction.
        \end{itemize}

        \item If $|y_1|= |y_2|$, then $t_1z_1 t_1 z_1 t_1 = \theta(z_2)t_2 z_2t_2 z_2$ or $t_1z_1 t_1 z_1 t_1 = z_2 t_2 z_2t_2 z_2$ or $t_1z_1 t_1 z_1 t_1 = \theta(z_2)t_2 \theta(z_2) t_2 z_2$, which are not possible as described above.
           \end{itemize}

         \item \textbf{Case 3.2}  Let $t_1z_1 t_1 z_1 t_1 = (z_2t_2)^{i'} z_2$ for some $i' \geq 0$. Then similar to Case 3.1, we have a contradiction.
         

           \item \textbf{Case 3.3}  Let $t_1z_1 t_1 z_1 t_1 = (t_2 \theta(z_2))^{i'} (t_2z_2)^{j'} t_2$ for some $i', j' \geq 0$. 
           \begin{itemize}
               \item 
           If $|y_1|< |y_2|$, then $t_1z_1 t_1 z_1 t_1 = t_2$ or $t_1z_1 t_1 z_1 t_1 = t_2z_2 t_2$ or 
          $t_1z_1 t_1 z_1 t_1 = t_2 \theta(z_2) t_2$ or $t_1z_1 t_1 z_1 t_1 = t_2 \theta(z_2) t_2z_2 t_2$
          or $t_1z_1 t_1 z_1 t_1 = t_2 \theta(z_2) t_2 \theta(z_2) t_2$ or $t_1z_1 t_1 z_1 t_1 = t_2 z_2 t_2z_2 t_2$.
          
          If $t_1z_1 t_1 z_1 t_1 = t_2$, then $z_1 \in \textup{P}_\theta$, a contradiction.

          Let  $t_1z_1 t_1 z_1 t_1 = t_2z_2 t_2$. Then $H(t_1z_1 t_1 z_1 t_1, t_1 \theta(z_1) t_1 \theta(z_1) t_1)=2$ which is a contradiction as $H(t_2z_2 t_2, t_2 \theta(z_2) t_2)=1$. Similarly $t_1z_1 t_1 z_1 t_1 = t_2 \theta(z_2) t_2$ is not possible. 
          
          
          If $t_1z_1 t_1 z_1 t_1 = t_2 \theta(z_2) t_2 \theta(z_2) t_2$ or $t_1z_1 t_1 z_1 t_1 = t_2 z_2 t_2z_2 t_2$, then by Theorem \ref{FineandWilf}, $t_2 \theta(z_2) \notin Q$ or $t_2 z_2 \notin Q$, which is a contradiction.
          
          If $t_1z_1 t_1 z_1 t_1 = t_2 \theta(z_2) t_2z_2 t_2$, then as $|t_1 z_1|< |t_2 z_2|$, we have $t_2 \theta(z_2) = t_1 z_1 \alpha$ and $z_2 t_2 = \alpha' z_1 t_1$ for some $\alpha, \alpha' \in \Sigma^+$. Then, $t_1z_1 t_1 z_1 t_1 = t_2 \theta(z_2) t_2z_2 t_2$  gives $\alpha' =\theta(\alpha)$. Now, $t_1z_1\alpha=t_2 \theta(z_2) =\theta(z_2 t_2) = t_1 \theta(z_1) \alpha$. This gives $z_1 \in \textup{P}_\theta$, a contradiction.
          
           \item 
           If $|y_1|= |y_2|$, then $t_1z_1 t_1 z_1 t_1 = t_2 \theta(z_2) t_2z_2 t_2$
          or $t_1z_1 t_1 z_1 t_1 = t_2 \theta(z_2) t_2 \theta(z_2) t_2$ or $t_1z_1 t_1 z_1 t_1 = t_2 z_2 t_2z_2 t_2$.
          If $t_1z_1 t_1 z_1 t_1 = t_2 \theta(z_2) t_2z_2 t_2$, then $z_2 \in \textup{P}_\theta$, a contradiction.
          If $t_1z_1 t_1 z_1 t_1 = t_2 z_2 t_2z_2 t_2$, then $y_1 = y_2$.
          If $t_1z_1 t_1 z_1 t_1 = t_2 \theta(z_2) t_2 \theta(z_2) t_2$, then $t_1=t_2$ and $z_1 = \theta(z_2)$. So, $y_1 = z_1 t_1$ and $y_2 = \theta(z_1) t_1$. Then it can be easily verified that $t_1 z_1 t_1 \theta(z_1) t_1$ nearly $\theta$-commutes with $y_2$, but not with $y_1$, hence $t_1 z_1 t_1 \theta(z_1) t_1 \in L_\theta(y_2) \setminus L_\theta(y_1)$, a contradiction.
           \end{itemize}
          \end{itemize}
          
    \end{itemize}
    The converse is straightforward.
\end{proof}

\section{Conclusions}
This manuscript primarily focused on properties of words $x$ and $y$ such that $H(xy, \theta(y)x)=1$. A characterization of $x$ and $y$ such that $H(xy, \theta(y)x)=1$ is given. The discussions on words $x$ and $y$ such that $x$ nearly $\theta$-commutes with $y$ in this manuscript have led to some interesting questions (as future work):

\textbf{Question 1:} What are the necessary and sufficient conditions on $x$ and $y$ such that the relation $R_\theta$ is transitive?

\textbf{Question 2:} For a given word $w$, how many distinct ways we can factorize $w$ in the form $tz$ and $zt$ where $t \in \textup{P}_\theta$ and $H(z, \theta(z))=1$?


\textbf{Question 3:} How many length $n$ words $w$ are possible such that $w$ has at least one factorization $xy$ where $x$ nearly $\theta$-commutes with $y$?

\bibliographystyle{abbrv}
\bibliography{splncs04.bib}

\end{document}